\documentclass[a4paper,11pt]{amsart}
\usepackage{graphicx,amssymb}
\usepackage{amsmath,amscd}
\usepackage[english]{babel}
\usepackage[applemac]{inputenc}
\usepackage[small,nohug,heads=littlevee]{diagrams}
\usepackage{hyperref}
\newtheorem{thm}{Theorem}[section]
\newtheorem{cor}[thm]{Corollary}
\newtheorem{lem}[thm]{Lemma}
\newtheorem{prop}[thm]{Proposition}
\newtheorem*{conjA}{Conjecture A}
\theoremstyle{definition}
\newtheorem{defn}[thm]{Definition}

\newtheorem{example}[thm]{Example}

\theoremstyle{remark}
\newtheorem{rem}[thm]{Remark}

\newcounter{numl}

\newcommand{\labelnuml}{\textup{(\roman{numl})}}
\newenvironment{numlist}{\begin{list}{\labelnuml}%
{\usecounter{numl}\setlength{\leftmargin}{0pt}%
\setlength{\itemindent}{2\parindent}%
\setlength{\itemsep}{\smallskipamount}\def
\makelabel ##1{\hss \llap {\upshape ##1}}}}{\end{list}}

\newcommand{\tauc}{{\mathcal T}}

\def\cal{\mathcal}
\def\bb{\mathbb} 

\def\a{\alpha } 
\def\b{\beta }
\def\g{\gamma }

\def\D{\Delta}

\def\g{\gamma }
\def\G{\Gamma }
\def\Z{\mathbb Z}
\def\R{\mathbb R}
\def\C{\mathbb C}
\def\N{\mathbb N}

\def\l{\lambda }

\def\O{\Omega }

\def\s{\sigma }

\def\deg{{\rm deg\ }}
\def\det{{\rm det}\ }

\def\omlog{\theta}

\def\ot{\otimes }
\def\part{\partial }
\def\bpart{\bar\partial }

\def\Re{{\rm Re}}
\def\Im{{\rm Im}}
\def\w{\wedge }

\def\Res{{\rm Res}}
\def\Pic{{\rm Pic}}
\def\cL{{\mathcal L}}

\begin{document}


\title[Lee classes on complex surfaces] {Twisted differentials  and Lee classes of locally conformally symplectic complex surfaces}

\author[V. Apostolov]{Vestislav Apostolov} \address{Vestislav Apostolov \\ Department of Mathematics\\
Laboratoire Jean Leray\\
2, Rue de la Houssini\`ere - BP 92208\\
F-44322 Nantes, FRANCE\\ and \\ Institute of Mathematics and Informatics\\ Bulgarian Academy of Sciences} 
\email{Vestislav.Apostolov@univ-nantes.fr}

\author[G. Dloussky]{Georges Dloussky}\address{Georges Dloussky, Aix-Marseille University, CNRS, Centrale Marseille, I2M, UMR 7373, 13453\\
39 rue F. Joliot-Curie 13411\\
Marseille Cedex 13, France}
\email{georges.dloussky@univ-amu.fr}

\thanks{V.A. was supported in part by a ``Connect Talent'' Grant of the R\'egion des Pays de la Loire in France. He is also grateful to the Institute of Mathematics and Informatics of the Bulgarian Academy of Sciences where a part of this project was realized. G.D. thanks the University of Nantes for their support and hospitality. 
The authors warmly thank S. Dinew and A. Zeriahi for their advise on the theory of PSH functions. V.A. is grateful to B. Chantraine for his interest in our work and  valuable discussions on the topology of LCS manifolds. }

\date{\today}

\begin{abstract} We study the set of deRham classes  of Lee $1$-forms of the locally conformally symplectic (LCS) structures taming the complex structure of a compact complex surface in the Kodaira class VII,  and show that the existence of non-trivial upper/lower bounds with respect to the degree function correspond respectively to the existence of certain negative/non-negative PSH functions on the universal cover. We use this to prove that the set of Lee deRham classes of taming LCS is connected, as well as to obtain an explicit negative upper bound for this set on the hyperbolic Kato surfaces. This leads to a complete description of the sets of Lee classes on the known examples of class VII complex surfaces, and to a new obstruction to the existence of bi-hermitian structures on the hyperbolic Kato surfaces of the intermediate type. Our results also reveal a link between  bounds of the set of Lee classes and non-trivial logarithmic holomorphic $1$-forms with values in a flat holomorphic line bundle. 
\end{abstract}

\dedicatory{To the memory of  Marco Brunella  whose work inspired us.}

\maketitle


\section{Introduction} The classification, up to biholomorphisms,  of the compact complex manifolds of complex dimension $2$ was reduced by the seminal works of Enriques and Kodaira to understanding the minimal models of 10 classes of complex surfaces~\cite[Tab.10]{bpv}. The description of these minimal models is essentially optimal in the case when the complex surfaces admit K\"ahler metrics, a purely topological condition  equivalent to  the first Betti number being even~\cite{buchdahl, lamari, siu}.  On the other hand, in the non-K\"ahler case, there is a still open classification conjecture regarding the minimal complex surfaces in the class {\rm VII}, i.e.  those having the first Betti number equal to $1$ and negative Kodaira dimension.  It is now known by the works of~\cite{D01, K77, nakamura-survey, andrei} that such a surface must be given by a list of \emph{known examples} if one can show that, when the second Betti number is positive, the surface admits a \emph{global spherical shell} (GSS). The existence of  such a GSS is often referred to as the \emph{GSS Conjecture} (see~\cite{nakamura-survey} and the references therein).  Whereas some structure results about the existence of compatible hermitian metrics, and especially locally conformally  K\"ahler hermitian metrics~\cite{Belgun, B1,B2, GO, FP19},  on the complex surfaces in the Kodaira class VII were obtained by assuming the GSS Conjecture, our thesis in this paper is that getting such structure results, and studying the geometric invariants associated to them, for possibly unknown surfaces is a promising angle of attack of the classification conjecture.

\smallskip
In the above general vein, this paper is a continuation of our previous work \cite{lcs},  in which we showed that \emph{any}  non-K\"ahler compact complex  surface $S=(M, J)$  admits  hermitian metrics  induced by locally conformally symplectic (LCS) $2$-forms $\omega$ which tame the underlying complex structure $J$. 

Recall that a LCS structure $\omega$ is a non-degenerate $2$-form on $M$,  which is closed in the twisted sense that
\[ d_{\alpha} \omega := d\omega - \alpha \wedge \omega=0, \]
where $\alpha$ is a closed $1$-form on $M$, called  the \emph{Lee form} of $\omega$. The taming condition means that the $(1,1)$-part $F:= (\omega)^{1,1}$ of $\omega$ with respect to $J$ is positive-definite, i.e. is the fundamental form of a hermitian structure $(g, J)$ on $S$. This definition is tailored to conformal geometry:  if  $\tilde \omega = e^{-f}\omega$ for some smooth function $f$, then $\tilde \omega$ is a taming LCS structure on $S$ with Lee form $\tilde \alpha = \alpha + df$,  and induced hermitian metric $\tilde g = e^{-f} g$.

\smallskip
 Given a compact complex surface $S$,  we introduce the following  subset of the first deRham cohomology group $H^1_{dR}(S, \R)$:
\[\tauc(S):=\{ [\alpha] \in H^1_{dR}(S, \R) \, | \, \alpha \, \textrm{is a Lee form of a taming LCS structure on }  S\}, \]
which is a natural complex geometric invariant of $S$.

\smallskip
The  fact that $\tauc(S)$ is non-empty contrasts the well-known obstruction to the existence of taming symplectic structures, which forces $b_1(S)$ to be even~\cite{gauduchon,lamari} and $S$ to be K\"ahler~\cite{buchdahl,lamari}.  In this latter case, it is  also observed in \cite{lcs} that $\tauc(S)=\{0\}$ reduces to a single trivial class. 

\smallskip
Thus motivated, we have initiated in \cite{lcs,lcs1} a systematic investigation of the set $\tauc(S)$ on the \emph{non-K\"ahler} compact complex surfaces.  This naturally extends the related invariant studied  in~\cite{Belgun,B1,B2,tsukada,OV}  in the case of  \emph{locally conformally K\"ahler hermitian metrics}, i.e. when $\omega$ is a LCS structure which not only tames the underlying complex structure but is also compatible with it.   However, by the results in \cite{Belgun},  not all compact complex surfaces admit locally {locally conformally K\"ahler hermitian metrics}, so the proposed extension is meaningful even for the known examples of non-K\"ahler complex surfaces.

\smallskip
The question of determining $\tauc(S)$ is particularly interesting and yet difficult to answer in the case when  $b_1(S)=1$ and the $H^0(S, K_S^m)=0, \, \forall m \ge 1$, i.e. on the complex surfaces in {the Kodaira class {\rm VII}}.  To state our results, we  first order $H^1_{dR}(S, \R)$  by the degree function  with respect to  some Gauduchon metric $g$ on $S$, i.e.  for $a, b\in H^1_{dR}(S, \R)$ we let $\cal L_{a}$ and $\cal L_b$ be the corresponding flat holomorphic line bundles  associated to the flat real line bundles determined via $a$ and $b$, and say that $a\le b$ (resp. $a<b$) if ${\rm deg}_g(\cal L_a) \le {\rm deg}_g(\cal L_b)$ (resp. if ${\rm deg}_g(\cal L_a) < {\rm deg}_g(\cal L_b)$). It is well-known~\cite{Teleman} that the relation $a\le b$  (resp. $a<b$) is independent of the choice of Gauduchon metric. In these terms, we have shown in \cite{lcs} that
\[ \tauc(S) \subseteq (-\infty, 0). \]
Furthermore,  in \cite{lcs1} we proved that $\tauc(S)$ is a non-empty \emph{open} subset,  unless $S$ is obtained by blowing up Inoue--Bombieri surface in which case $\tauc(S)=\{a\}$ is a single point. 

\smallskip
In this paper, we improve the above statements and establish the following general structure result:
\begin{thm}\label{thm:tau-structure} Let $S$ be a compact complex surface in the Kodaira class {\rm VII}. Then $\tauc(S)$ is either a single point, in which case $S$ is obtained from an Inoue--Bombieri surface by blowing points, or else  $\tauc(S)$ is an open interval
\[ \tauc(S) = (d_ 0, b_0) \subseteq (-\infty, 0).\]
\end{thm}
The main new idea in the proof of Theorem~\ref{thm:tau-structure} is a relation between the existence of upper and lower bounds for the set $\tauc(S) \subset (-\infty, 0)$ with the existence of, respectively,  negative and positive currents of degree $0$ on $S$, satisfying a twisted PSH condition, see Propositions~\ref{p:obstruction} and \ref{p:sign} below.  This new observation can be loosely stated as follows:

\begin{prop}\label{p:main}(see Proposition~\ref{p:obstruction} and Lemma~\ref{l:cover})  Suppose $S$ is a compact complex surface in the Kodaira class {\rm VII} and denote by  $\hat S$ the minimal $\Z$-cover of $S$ such that $H^1(\hat S, \R)=\{0\}$, see Definition~\ref{d:cyclic}.
\begin{enumerate}
\item[\rm (a)] There exists a deRham class $b<0$ such that $\tauc(S)\subset (-\infty, b)$ if and only if  $\hat S$  admits a  negative  strictly PSH function which is multiplicative automorphic with respect to the action of the cyclic generator of the covering subgroup in $Aut(\hat S)$,  with constant of automorphy $C>1$.  
\item[\rm (b)] There exists a deRham class $d<0$ such that $\tauc(S)\subset (d, 0)$ if and only if $\hat S$  admits a  non-negative  strictly PSH function which is multiplicative automorphic with respect to the action of the cyclic generator,  with constant of automorphy $C>1$. 
\end{enumerate}
\end{prop}
The proof of Proposition~\ref{p:main}  and the derivation of Theorem~\ref{thm:tau-structure} rely on relatively elementary methods which come in the form of theory of linear second order elliptic differential operators,  and applications of Hahn-Banach's theorem as in \cite{HL}.

\smallskip
Specializing to the \emph{known} examples of class VII complex surfaces, Proposition~\ref{p:main}  and Theorem~\ref{thm:tau-structure}  can be used to obtain the following sharp classification result:

\begin{thm}\label{main} Suppose $S$ is a compact complex surface obtained from a known minimal complex surface in the Kodaira class {\rm VII},  by blowing up points. 
 \begin{enumerate}
 \item[\rm (a)] If $S$ is obtained by  blowing up points of an Inoue--Bombieri surface,  then $\tauc(S)=\{a\} \subset (-\infty, 0)$ consists of a single point;
 \item[\rm (b)] If $S$ is obtained by blowing up points of a Hopf surface, then $\tauc(S)=(-\infty, 0)$;
 \item[\rm (c)] If $S$ is obtained by  blowing up points of a hyperbolic Kato surface, then  $\tauc(S)=(-\infty, b)$ for some $b<0$;
 \item[\rm (d)]  If  $S$ is obtained by  blowing up points of an Enoki surface, then $\tauc(S)=(-\infty, 0)$. 
 \end{enumerate}
 \end{thm}
 In the above statement, the surfaces listed in (c) and (d)  are the minimal  class VII complex surfaces with a GSS constructed by M. Kato~\cite{K77,D84},  and  \emph{hyperbolic Kato surfaces}  refer to the Kato surfaces of intermediate type or the Inoue--Hirzebruch surfaces in the terminology of \cite{Brunella14}.  As we have mentioned, (a)  follows from \cite{lcs1} whereas (b) is established in \cite{lcs,tsukada}, using the constructions of \cite{Belgun, GO}. The non-boundedness from below of $\tauc(S)$ in the cases (c) and (d) is due to Brunella~\cite{B1}.  To the best of our knowledge,  (c) provides the first known constraint on $\tauc(S)$,   other than the inclusion $\tauc(S)\subset (-\infty, 0)$ established in \cite{lcs}. In this case, our result is new even for the set of Lee classes of the locally conformally K\"ahler metrics, and  in fact it partially answers questions raised in \cite[Rem.9]{B1}. Similarly, (d) provides the first proof that $\tauc(S)=(-\infty, 0)$ on the Enoki surfaces. It will be interesting to know whether or not  in the cases (c) and (d) of the above theorem, $\tauc(S)$ coincides with the set of Lee classes of the locally conformally K\"ahler metrics on $S$.

\smallskip
The existence of PSH functions on $\hat S$ as in Proposition~\ref{p:main} was studied by Brunella~\cite{Brunella13,Brunella14} who showed that if such a PSH function has also \emph{analytic singularities}, then $S$ must be a known complex surface (see also Section~\ref{s:conjecture} below). We thus expect that when $\tauc(S) \subsetneq (-\infty, 0)$ in Theorem~\ref{thm:tau-structure}, one would be able to reduce the situation to the one studied by Brunella, and propose the following
\begin{conjA}\label{c:main} Suppose $S$ a compact complex surface in the Kodaira class {\rm VII}. 
\begin{enumerate}
\item[\rm (a)] If there exists a  deRham class $b\in (-\infty, 0)\subset H^1_{dR}(S, \R)$ such that $\tauc(S)\subset (-\infty, b)$, then  $S$ is obtained from either an Inoue--Bombieri or a hyperbolic Kato surface by blowing up points.
\item[\rm (b)] If there exists a  de Rham class $d\in (-\infty, 0) \subset H^1_{dR}(S, \R)$ such that $\tauc(S) \subset (d, 0)$, then $S$ is obtained from an Inoue--Bombieri  surface by blowing up points.
\end{enumerate}
\end{conjA}
Of course, the above conjecture will be an immediate corollary of Theorem~\ref{main} if one assumes the Global Spherical Shell Conjecture~\cite{Nakamura}, but we would rather think about it as a plausible intermediate step in deriving a general classification. 

\smallskip
To illustrate the usefulness of Theorem~\ref{main},  recall that~\cite{Nakamura}  the minimal complex surfaces with a GSS  and  second Betti number equal to $k$ are all diffeomorphic to $M_k:=(S^1\times S^3) \sharp k \overline{\C P^2}$.  The underlying smooth manifold $M_k$  then admits three types of complex structures  with a  GSS,  described in the cases (b), (c) and (d) of  Theorem~\ref{main}. As  the taming condition is open in the $C^{\infty}$ topology of almost complex structures,  we then get from Theorem~\ref{main} the following stability  result for hyperbolic Kato surfaces under the action of the diffeomorphism group:
\begin{cor}\label{c:Cheeger-Gromov} Suppose $S=(M_k, J)$ is a hyperbolic Kato surface with $k=b_2(S)$, and $(\phi_i) \subset {\rm Diff}_0(M_k)$ is a sequence of diffeomorphisms in the connected component of the identity,  such that $\phi_i(J)$ converges in $C^{\infty}$ to an integrable complex structure $J_{\infty}$ on $M_k$.  If $S_{\infty}:=(M_k, J_{\infty})$ admits a GSS, then it must be a hyperbolic Kato surface.
\end{cor}
The above corollary is  potentially useful in the application of analytic methods,  such as the Streets--Tian theory of  the pluriclosed flow~\cite{Streets}. By  the results in \cite{DK98},  Corollary~\ref{c:Cheeger-Gromov} holds true in the setting of \emph{analytic  degenerations}  of $J$.

\smallskip
We  also find an \emph{explicit} upper bound $b$  for $\tauc(S)$ on the  hyperbolic Kato surfaces,  by relating the corresponding deRham class $b$ with the existence of a non-zero logarithmic holomorphic $1$-form with coefficients in the flat holomorphic line bundle determined by $b$ (see Section~\ref{s:explicit-bound} and Lemma~\ref{ex:log-holom}).  As another   application of Theorem~\ref{main}, we return to our initial motivation of studying LCS structures on class {\rm VII} surfaces~\cite{ABD},  which stems from the classification of \emph{bi-hermitian structures} on these complex surfaces~\cite{ABD,FP19,SY}. We  give now a new obstruction  to the existence of bi-hermitian structures  on  Kato surfaces.
\begin{thm}\label{thm:main-BH} (see Theorem~\ref{thm:BH}) There are intermediate type hyperbolic Kato surfaces $S$, satisfying $H^0(S, K^{-1}_S \otimes {\mathcal L}_{b})\neq 0$ where ${\mathcal L}_{b}$ is the flat holomorphic line bundle associated to a deRham class $b\in (-\infty, 0)$, but which do not admit any bi-hermitian structure.
\end{thm}
Our explicit construction on the hyperbolic Kato surfaces leads us to investigate in the final Section~\ref{s:TLF}  of the paper the general problem of   existence of logarithmic $1$-forms with values in a flat line bundle.  We recall that any such form determines a holomorphic foliation $\cal F$ on $S$ and, the \emph{characteristic numbers} of a non-degenerate singularity $p$ of $\cal F$ are given by the two ratios of eigenvalues of the linear part of any local holomorphic vector field generating $\cal F$ around $p$ and vanishing at $p$. We prove the following  result of independent interest.
\begin{thm}\label{thm:logarithmic-forms} Suppose $S$ a {minimal} compact complex surface in the Kodaira class {\rm VII},  which admits a non-zero holomorphic logarithmic $1$-form $\theta$  with values in a flat holomorphic line bundle $\cL_b$  obtained as the complexification of a real flat line bundle corresponding to a deRham class $b\in H^1_{dR}(S, \R)$. Suppose, furthermore, that ${\rm deg}_g (\cal L_b) \le0$ with respect to some (and hence any) Gauduchon metric $g$ on $S$, and that the holomorphic foliation $\cal F$ defined by $\theta$ has real negative characteristic  numbers at any non-degenerate singular point situated on the pole of $\theta$. Then $b$ bounds $\tauc(S)$ and one of the following must hold:
\begin{enumerate}
\item[\rm (a)] $b=0$ and $\tauc(S) = (-\infty, 0)$. In this case  $S$ is either a Hopf surface or an Enoki surface.
\item[\rm (b)] $b<0$ and $\tauc(S) \subseteq (-\infty, b)$. In this case $S$ is a hyperbolic Kato surface.
\item[\rm (c)] $b<0$ and $\tauc(S) =\{b\}$. In this case $S$ is an Inoue--Bombieri surface.
\item[\rm (d)] $b<0$ and $\tauc(S) \subseteq (b, 0)$.  
\end{enumerate}
\end{thm}
All known minimal complex surface in the Kodaira class {\rm VII}, except possibly some secondary Hopf surfaces,  do appear in the cases (a)-(c) of the above statement (see Remark~\ref{r:positive-type} below) whereas the case (d) should not occur if Conjecture A-(b) holds true. We know of no example $(S, \theta)$ such that $\theta$ is a non-zero holomorphic logarithmic $1$-form $\theta$  with values in a flat holomorphic line bundle $\cL_b$  with ${\rm deg}_g (\cal L_b) \le 0$,   and for which the characteristic numbers of $\cal F$ on the pole are not all negative. However, positive characteristic numbers  do occur on some  Inoue--Hirzebruch  surfaces~\cite{DO} if we drop the assumption  ${\rm deg}_g (\cal L_b)\le 0$ for $\cal L_b$. 
It is interesting to know whether or not in the case (b) we have equality $\tauc(S)=(-\infty, b)$.

\smallskip The article is structured as follows. In Section~\ref{s:twisted-differentials},  we introduce differential forms and currents with values in flat real line bundles, the corresponding twisted differentials,   twisted deRham and Dolbeault cohomology groups, and the duality between them. Section~\ref{s:key} specializes  to  non-K\"ahler compact complex surfaces, and is the main technical body of the paper. Here we recast the basic constructions from \cite{lcs} which we use in this paper, and establish in Proposition~\ref{p:conformal-dual} a Harvey--Lawson's type criterion for the existence of a taming LCS structure with a given Lee form. This is then  enhanced, via linear elliptic theory, to sharper statements (Propositions~\ref{p:obstruction} and \ref{p:sign}) which express the bounds for the set $\tauc(S)$ in terms of the existence of definite currents of degree zero and with positive twisted $i\partial \bar \partial$-image. These lead to the proof of the connectedness of $\tauc(S)$ (see Theorem~\ref{thm:tau-connected}) which, together with some previous results from \cite{lcs,lcs1}, yield the proof of Theorem~\ref{thm:tau-structure} from the Introduction. In the next Section~\ref{s:examples}, we recall the classification of the known minimal complex surfaces in the class VII,  and examine on each class of them the existence of definite degree zero currents with positive twisted $i\partial \bar \partial$-image,  leading to the proof of Theorem~\ref{main} and its  corollaries, Corollary~\ref{c:Cheeger-Gromov} and  Theorem~\ref{thm:main-BH}. Section~\ref{s:conjecture} recasts some known results,  due to Brunella~\cite{Brunella13,Brunella14} and \cite{lcs1},  in the setup of this article. They further motivate Conjecture~A from the introduction,  and are used in the next Section~\ref{s:logarithmic} in which we review the theory of  holomorphic logarithmic  $1$-formes with values in a flat bundle and prove Theorem~\ref{thm:logarithmic-forms} from the Introduction.

\section{Twisted differentials and currents}\label{s:twisted-differentials}

Let  $X=(M,J)$ be a compact complex manifold of complex dimension $n$, $\alpha$ a closed $1$-form  on $M$, representing a de Rham class $a:=[\alpha] \in H^1_{dR}(M, \R)$. We denote by $L_{\alpha}=M\times \R$ the topologically trivial real line bundle over $M$,  endowed with the flat connection $\nabla^{\alpha} = d + \alpha$, and by $\cal L_{\alpha}= L_{\alpha}\otimes \C$ the corresponding topologically trivial flat holomorphic line bundle.  This construction fits in the sequence 
\begin{equation}\label{H1}
H^1_{dR}(M,\bb R)\stackrel{\exp}{\hookrightarrow} H^1_{dR}(M, \bb R^*_+) \longrightarrow  H_0^1(M,\bb C^*) \longrightarrow {\Pic}^0(X), 
\end{equation}
where $H_0^1(M, \C^*)$ and $\Pic^0(X)$ respectively denote the spaces of equivalent classes of flat and holomorphic \emph{topologically trivial} complex line bundles.  We shall slightly abusively denote by $L=L_{a}$ the isomorphism class of topologically trivial flat real bundles corresponding to $a\in H^1_{dR}(M, \R)$,  and  by
$L^*= L_{-a}$ its dual. The corresponding flat holomorphic line bundles over $(M,J)$ will be denoted by $\cal L = L_a \otimes \C$ and $\mathcal{L}^*= \cal L_{-a}$, respectively.

We denote by  $\mathcal{E}^k(M,\R)$ (resp. $\mathcal{E}^{p,q}(X, \C)$) the sheaf of  smooth real $k$-forms on $M$ (resp. of  complex-valued $(p,q)$-forms on $X$).

For any $\alpha \in a$, the differential operator  $d_{\alpha}:=d-\alpha\wedge \cdot$  on $\cal E^{k-1}(M, \R)$  corresponds to the induced differential  on $\cal E^k(M, L^*)$ via the flat connection $\nabla^{-\alpha} := d -\alpha$ on $L^*=L_{-a}$. Furthermore, $d_{\alpha}$ defines the Lichnerowicz--Novikov complex
\begin{equation}\label{lichne-novikov}
\cdots \stackrel{d_{\alpha}}{\to}\cal E^{k-1}(M, \R) \stackrel{d_{\alpha}}{\to} \cal E^{k}(M, \R)  \stackrel{d_{\alpha}}{\to}\cdots 
\end{equation}
which is  isomorphic to the de Rham complex of differential forms with values in $L^*$
\begin{equation}\label{twisted-deRham}
\cdots \stackrel{d_{L^*}}{\to}\cal E^{k-1}(M,L^*) \stackrel{d_{L^*}}{\to} \cal E^{k}(M, L^*)  \stackrel{d_{L^*}}{\to}\cdots.
\end{equation}
This isomorphism can be made explicit by writing $\alpha_{\mid U_i}=df_i$ on an open covering  $\mathfrak {U}=(U_i)$ of $M$: then, for any  $d_{\alpha}$-closed smooth form $\omega$ on $M$,   ${\omega_i}_{\mid U_i} := e^{-f_i}\omega$ gives rise to a $d_{L^*}$-closed form with values in $L^*$.

In particular,  we have an isomorphism of  cohomology groups
$$H^{k}_\a(M, \R)\simeq H^k_{d_{L^*}}(M,L^*).$$
For the Dolbeault cohomology groups of $X$ with values in the flat holomorphic line bundle ${\mathcal L}^*$, we have  similarly
$$d_{L^*}=\part _{\mathcal {L^*}}+\bpart_{\mathcal {L}^*} \quad {\rm and}\quad d_{\alpha}=\part_{\alpha}+\bpart_{\alpha},$$
 where 
 $$\part_{\alpha}=\part-\alpha^{1,0}\w\quad {\rm and}\quad \bpart_{\alpha}=\bpart-\alpha^{0,1}\w,$$
thus giving rise to the isomorphisms 
\begin{equation}\label{dolbeault}
H^{p,q}_{\bar \partial_{\alpha}}(X, \C)\simeq H_{\bar\partial_{\cal L^*}}^{p,q}(X,\cal L^*).
\end{equation}
We can further consider the Aeppli  and  the Bott--Chern cohomology groups of $(p,q)$-forms with values in $\mathcal{L}^*$,  respectively denoted $H^{p,q}_A(X, \cal L^*)$ and $H^{p,q}_{BC}(X, \cal L^*)$, also seen as the spaces of $\partial_{\a}\bar\partial_{\a}$-closed $(p,q)$-forms modulo ${\rm im}(\partial_{\a}) \oplus {\rm im}(\bar \partial _{\a})$ (resp. $d_{\a}$-closed $(p,q)$-forms modulo $\partial_{\a}\bar\partial_{\a}$-exact ones). The Hodge duality theory (see e.g. \cite[Corollary~1.5]{AK}) yields in this case
\begin{lem}\label{l:duality} Let $X=(M,J)$ be a compact complex manifold of complex dimension $n$ and ${L}$ a real flat line bundle. Then, the natural pairing through integration of wedge products of $2$-forms  with values in $L$ and $L^*$ respectively, gives rise to  canonical  isomorphisms 
$$\Big(H_{d_L}^{k}(M, {L})\Big)^* \cong H_{d_{L^*}}^{2n-k}(M, {L}^*), \ \ \Big(H^{p,q}_{BC}(X, \cal L)\Big)^* \cong H^{n-p,n-q}_{A}(X, \cal L^*),$$
where the upper  $*$ denotes the dual vector space.
\end{lem}

For any open subset $U\subset M$, we denote by $\mathcal{E}_c^k(U, \R)$ and $\mathcal{E}_c^{p,q}(U, \C)$ the respective spaces of smooth forms with compact support in $U$. The space $\mathcal{D}'_k(U, \R)$ of currents of dimension $k$ (and degree $2n-k$) is the dual topological space of $\mathcal{E}^k_c(U)$. Similarly, the space $\mathcal{D}'_{p,q}(U, \C)$ of currents of bi-dimension $(p,q)$ (and  bi-degree $(n-p, n-q)$) is the dual topological space   of $\mathcal{E}_c^{p,q}(U, \C)$.  We shall use the standard notation  $\mathcal{D}'^{r}(U, \R) := \mathcal{D}'_{2n-r}(U, \R)$ and $\mathcal{D}'^{r,s}(U, \C):= {\mathcal D}'_{n-r, n-s}(U, \C)$ and  denote by $j$ the natural injections $j : \mathcal{E}^k(M, \R) \hookrightarrow \mathcal{D}'^k(M, \R)$ and $j: \mathcal{E}^{p,q}(X, \C) \hookrightarrow \mathcal{D}'^{p,q}(X, \C)$ obtained by the wedge product and integration.

\smallskip

The action of the operator $d_{\alpha}: \mathcal{D}'_{k} (M, \R)\mapsto \mathcal{D}'_{k-1}(M, \R)$ on currents is defined by
\begin{equation}\label{d-current-1}
(d_{\alpha} T)(\phi) := (-1)^{k-1} T(d_{\alpha}\phi), \ \forall \phi \in \mathcal{E}^{k-1}(M, \R),
\end{equation}
and similarly for the operators 
$$\partial_{\alpha} : \mathcal{D}'_{p,q}(X, \C) \to \mathcal{D}'_{p-1, q}(X, \C), \ \ \bar \partial_{\alpha} : \mathcal{D}'_{p,q}(X, \C) \to \mathcal{D}'_{p, q-1}(X, \C).$$
Notice that for any $(k-1)$-form $\phi$ and $(2n-k)$-form $\psi$ on $M$, we have
\begin{equation}\label{star}
d (\phi \wedge \psi) = (d_{\alpha} \phi) \wedge \psi  + (-1)^{k-1} \phi \wedge d_{-\alpha}\psi.
\end{equation}
Thus,  the sign convention in \eqref{d-current-1} is such that for any smooth $(2n-k)$ form $\psi$ on $M$,  giving rise to a current $T_{\psi}:= j(\psi) \in \mathcal{D}'_k(M, \R)$, we have  
\begin{equation}\label{by-parts}
d_{\alpha} T_{\psi} =  T_{(d_{-\alpha} \psi)}, \ \  d_{-\alpha} T_{\psi} =-T_{(d_{\alpha} \psi)}.
\end{equation}
 Let $H^{\alpha}_{k}(M, \R)$, $H^{BC,\alpha}_{p,q}(X, \C)$, $H^{A,\a}_{p,q}(X, \C)$ be the corresponding homology groups in terms of currents,  defined respectively as $d_{\a}$-closed modulo $d_{\a}$-exact currents of degree $2n-k$, $d_{\a}$-closed modulo $\partial_{\a}\bar\partial_{\a}$-exact currents of bi-degree $(n-p,n-q)$, and $\partial_{\a}\bar\partial_{\a}$-closed currents of bi-degree $(n-p,n-q)$ modulo ${\rm im}(\partial_{\a}) \oplus {\rm im}(\bar \partial _{\a})$. The injection $j$ defined above gives rise to linear isomorphisms  
\[
\begin{split}
H^{2n-k}_{-\alpha}(M, \R)  &\cong H^{\alpha}_k(M, \R), \\
 H^{n-p, n-q}_{BC,-\alpha}(X, \C)  &\cong H^{BC,\alpha}_{p,q}(X, \C), \\
   H^{n-p, n-q}_{A,-\alpha}(X, \C) & \cong H^{A,\alpha}_{p,q}(X, \C).
   \end{split}
   \]
Combined with Lemma~\ref{l:duality}, we have canonical isomorphisms
\[
\begin{split}
H^{\alpha}_k(M, \R) &\cong (H^k_{\alpha}(M, \R))^*, \\
H^{BC,\a}_{p,q}(X, \C) &\cong (H^{p,q}_{A, \a}(X, \C))^*, \\
 H^{A,\a}_{p,q}(X, \C)&\cong (H^{p,q}_{BC, \a}(X, \C))^*.
 \end{split}
 \]

\smallskip
We use  the usual notions of  {\it positivity} of $T\in \cal D'_{p,p}(U, \R)$, defined as follows
\begin{defn} $T \in \cal D'_{p,p}(U, \R)$  is a {\it weakly positive} (or \emph{positive}) $(p,p)$-current ($T \ge 0$) if for each $\xi_1, \ldots \xi_p \in \mathcal{E}_c^{1,0}(U)$, 
$T( i\xi_1\wedge \bar \xi_1\wedge \cdots \wedge i\xi_p \wedge \bar \xi_p)\ge 0;$  the  current $T$ is called {\it strongly positive}  if $T - j(\psi^{n-p}) $ is a positive $(p,p)$-current, where $\psi$ is a positive-definite smooth $(1,1)$-form on $X$.
\end{defn}

\section{Twisted exact currents and LCS structures on complex surfaces}\label{s:key}
We now specialize to the case when $X=S$ is a compact complex surface with odd first Betti number $b_1(S)$, in which case we have  the following result: 

\begin{thm}\label{thm:conformal}\cite{lcs} Let $S=(M,J)$ be a compact complex surface with odd first Betti number,  and $\mathsf{c}$ a conformal class of hermitian metrics on $S$. Then, there exists a non-zero de Rham class $a \in H^1_{dR}(M, \R)$ such that for any metric $g \in \mathsf{c}$, there exists a representative $\alpha \in a$ such that the fundamental $2$-form $F$ of $g$ satisfies
\begin{equation}\label{twisted-pluriclosed}
d_{\alpha} d^{c}_{\alpha} F =0,
\end{equation}
where $d_{\alpha}^c: = J d_\alpha J^{-1}$.\end{thm}
The following  can be obtained by a simple and standard modification of the arguments in \cite{HL,otiman}.
\begin{prop}\label{p:conformal-dual} The complex surface $S=(M,J)$ admits a $d_{\alpha}d^c_{\alpha}$-closed positive definite $2$-form $F$ if and only if $S$ admits no non-trivial positive current $T \ge 0$ which is $d_{\alpha}d^c_{\alpha}$-exact.
\end{prop}
\begin{proof} $(\Longrightarrow)$ If there exists $\tau \in \mathcal{D}'_4(S,\R)$ such that $T=d_{\alpha}d^c_{\alpha}\tau$, then $T(F)=\tau(d_{\alpha}d^c_{\alpha}F)=0$. As $F>0$, it follows that $T=0$.

\smallskip
$(\Longleftarrow)$   The closed mapping theorem asserts that the adjoint of a linear mapping with a closed range has a closed range (see e.g. \cite{schaeffer}, Ch.IV, Sec.7.7). As $H^{BC, \a}_{1,1}(S, \C) \cong H^{1,1}_{BC, -\a}(S, \C) \cong H^{1,1}_{BC}(S, \cal L^*)$ is finite dimensional,  a similar argument as the one in \cite[p.~174]{HL} shows that the subspace $\mathcal C=\{T=d_{\alpha}d^c_{\alpha}\tau, \tau\in \mathcal{D}'_4(S, \R)\}$ is a closed subspace in ${\cal D'}_{1,1}(S, \R)$.  Let $\psi>0$ be a positive definite $(1,1)$ form and consider the weakly  compact convex subset $\mathcal K=\{T\in {\cal D'}_{1,1}(S, \R)\mid  T\ge  0, T(\psi)=1\}$. We have $\mathcal C\cap \mathcal K=\emptyset$ by the assumption. The spaces $\cal E^{p,q}(S, \C)$ and $\cal D'_{p,q}(S, \C)$ being dual of each other \cite[p.~75]{schwartz}, Hahn--Banach separation theorem implies that there exists a smooth $(1,1)$-form $F$ such that
\[  F_{\mid \mathcal C}=0 \, \, \textrm{and} \, \,  F_{\mid \mathcal K}>0. \]
The first condition implies that for each $\tau\in {\cal D'}_4(S, \R)$, 
$d_{\alpha}d^c_{\alpha}\tau(F)=\tau(d_{\alpha}d^c_{\alpha}F)=0,$
hence $d_{\alpha}d^c_{\alpha}F=0$,  whereas the second condition yields $F>0$. \end{proof}

Furthermore, as shown in \cite[Proposition~2.12 \& Lemma~2.13]{lcs}, the relevance of the condition \eqref{twisted-pluriclosed} is given by the following

\begin{prop} \label{p:lcs-characterization} On a compact complex surface $S$  with $b_1(S)=1$,  Kodaira dimension ${\rm kod}(S)=-\infty$, and whose minimal model has positive second Betti number, the following conditions are equivalent.
\begin{enumerate}
\item[\rm (i)] $0\neq a \in H^1_{dR}(S, \R)$ is such that there exists a hermitian  metric $g$ on $S$ and a representative $\alpha \in a$,  so that fundamental $2$-form $F$ of $g$ satisfies \eqref{twisted-pluriclosed};
\item[\rm (ii)]  $0\neq a \in H^1_{dR}(S, \R)$ is such that for any $\alpha \in a$, there exists a hermitian  metric $g$ on $S$  whose  fundamental $2$-form $F$ of $g$ satisfies \eqref{twisted-pluriclosed};
\item[\rm (iii)]  $0\neq a \in H^1_{dR}(S, \R)$ is such that for any $\alpha \in a$,  $S$ admits a locally conformally symplectic $2$-form $\omega$ which tames $J$ and satisfies $d_{\alpha} \omega  =0$. 
\end{enumerate}
\end{prop}
\begin{rem}\label{r:weak} On \emph{any} complex manifold $X=(M, J)$ endowed with a LCS structure $\omega$ taming $J$ and with Lee form $\alpha$, the fondamental $2$-form $F= \omega^{1,1}>0$ of the corresponding hermitian structure satisfies \eqref{twisted-pluriclosed}. The above Proposition provides a converse.
\end{rem}
\begin{defn} On a complex surface $S$  we denote by $\tauc(S)\subset H^1_{dR}(S, \R)$ the set of deRham classes  of the Lee forms $\alpha$ of locally conformally symplectic  $2$-forms $\omega$ which tame the complex structure $J$.
\end{defn}
By \cite[Theorem~1]{lcs}, on any compact complex surface  $\tauc(S) \neq \emptyset$. Furthermore, when $b_1(S)=1$, we can define an  order on $H^1_{dR}(S, \R)$  by the degree of  a  flat holomorphic line bundle ${\cal L}_{a}$  corresponding to $a\in H^1_{dR}(S, \R)$ via \eqref{H1}, with respect to some (and hence any \cite {Teleman}) Gauduchon metric on $S$. Then,  \cite[Proposition~4.3]{lcs} yields
\begin{prop}\label{p:tau}  On a compact complex surface $S$  with $b_1(S)=1$ and ${\rm kod}(S)=-\infty$ we have $\emptyset\neq \tauc(S) \subset (-\infty,  0)$.
\end{prop}
We are now going to establish the following refinement of Proposition~\ref{p:conformal-dual} in the case of complex surfaces with $b_1(S)=1$ and ${\rm kod}(S) =-\infty$.
\begin{prop}\label{p:obstruction} Let $S$ be a compact complex surface  with $b_1(S)=1$ and  Kodaira dimension ${\rm kod}(S)=-\infty$. 
\begin{enumerate}
\item[\rm (a)] Suppose $\beta \in b \in H^1_{dR}(S, \R)$ is such that there exists a  weakly negative current $\tau \le 0$ of degree $0$,  with  $T:= d_{\beta}d^c_{\beta}\tau \ge 0$  and $T\neq 0$. Then  $b\in (-\infty, 0)  \, \, \textrm{and} \, \, \tauc(S)\subset (-\infty, b).$\\
 Conversely, if $\tauc(S) \subset (-\infty, b)$ for some $b<0$, then  for each $\beta \in b$ there exists a  weakly negative current $\tau \le 0$ of degree $0$,  with  $T:= d_{\beta}d^c_{\beta}\tau \ge 0$  and $T\neq 0$. 
\item[\rm (b)] Suppose $\delta \in d \in H^1_{dR}(S, \R)$ is such that there exists a weakly positive current $\tau \ge 0$ of degree $0$,  with  $T:= d_{\delta}d^c_{\delta}\tau \ge 0$  and $T\neq 0$.  Then either $d \in (0, \infty)$ or else $\tauc(S) \subset (d, 0)$.\\
 Conversely, if $\tauc(S) \subset (d, 0)$ for some  de Rham class $d <0$,  or if $d>0$,  then  for each $\delta \in d$ there exists a weakly positive current $\tau \ge 0$ of degree $0$,  with  $T:= d_{\delta}d^c_{\delta}\tau \ge 0$  and $T\neq 0$. 
\end{enumerate}
\end{prop}
\begin{proof} (a)  We first notice that the condition $\tau \le 0$ is conformally invariant and $$d_{\beta} d^c_{\beta} \big(e^{f} \tau\big) = e^{f} \big(d_{\tilde \beta} d^c_{\tilde \beta} \tau\big)$$ with $\tilde \beta =\beta - df$. Thus,  after making a conformal modification of $T$  we can assume that $T= d_{\beta} d^c_{\beta} \tau \ge0, \, T\neq 0$ for some weakly negative degree zero current $\tau$ and  any given $\beta \in b$.

Let  $a\in \tauc(S)$. By Remark~\ref{r:weak}, there is an $\alpha\in a$ and a hermitian metric $g$ on $S$ whose fundamental $2$-form $F$ satisfies \eqref{twisted-pluriclosed}.  By changing $g$ in its conformal class  and $\alpha$ in its deRham class $a$, we can assume that $g$ is a Gauduchon metric~\cite{gauduchon-0} on $S$. Writing $b = ta$ for some $t>0$, we let $\beta:=t\alpha \in b$. Denoting  by $\theta_g= J \delta^g F$ be the Lee form of $g$, we compute as in \cite[Lemma~2.5]{lcs}:
\begin{equation}\label{computation}
\begin{split}
d_{\beta}d^c_{\beta} F = & \Big(d(J\theta_g -J\beta) + (\theta_g-\beta)\wedge J(\theta_g -\beta)\Big)\wedge F \\
                                     =& \frac{1}{2}\Big(\delta^g(\beta-\theta_g) + \big\langle(\beta-\theta_g), \beta\big\rangle_g\Big) F\wedge F \\
                                     =&\frac{1}{2}\Big(\delta^g\beta + \big\langle(\beta-\theta_g), \beta\big\rangle_g\Big) F\wedge F  \\
                                     =&\frac{t}{2}\Big(\delta^g\alpha + \big\langle(t\alpha-\theta_g), \alpha \big\rangle_g\Big) F\wedge F \\
                                     =& \frac{(t-1)t}{2}|\alpha|_g^2F\wedge F,
                                      \end{split}
                                     \end{equation}
                                     where for passing from the first line to the second we use the computation appearing in the proof of  \cite[Lemma~2.5]{lcs}, from the  second line to the third we we have used that $g$ is Gauduchon (i.e. $\delta^g \theta_g=0$) and  for passing from the forth line to the fifth we have used the identity
                                     \[\delta^g \alpha + \langle \alpha -\theta_g, \alpha \rangle_g=0\]
                                     which follows from Lemma 2.5 in \cite{lcs} and the fact that $g$ is Gauduchon.
                                      Applying $T=d_{\beta}d^c_{\beta} \tau$ to  $F$ (with $\beta = t\alpha$), we get
                                      \begin{equation}\label{inequality}
                                     0 < T(F) = \frac{t(t-1)}{2}\tau\Big(|\alpha|^2\ F\wedge F\Big).
                                     \end{equation}
                                     As $\tau\Big(|\alpha|^2{F\wedge F}\Big)\le 0$ by assumption, we obtain $t(t-1) < 0$, i.e. $t\in (0, 1)$. By Proposition~\ref{p:tau},  $t>0$ is equivalent to  $b\in (-\infty, 0)$; furthermore,  $\alpha = \beta/t < \beta$.   
                                     
 \smallskip
We now establish the other direction. We thus assume that $S$ is a  compact complex surface in the Kodaira class {\rm VII} and  $b\in (-\infty, 0)\subset H^1_{dR}(S, \R)$ is  such that $\tauc(S) \subset (-\infty, b)$.  We are going to show that  for any  smooth closed $1$-form $\beta \in b$,   there exists a weakly negative current $\tau\le 0$ of degree $0$ on $S$,  such that $T:=d_{\beta} d^c_{\beta}\tau \ge0$ and $T \neq 0$.

\smallskip
We first assume that the minimal model of $S$ has non-zero second Betti number, i.e. we are under the hypothesis of Proposition~\ref{p:lcs-characterization}.  In this case, $b$ equivalently bounds  the de Rham classes $a$ for which \eqref{twisted-pluriclosed} holds for some hermitian metric on $S$.
For $\beta \in b$ and a positive definite smooth $(1,1)$-form  $\psi$ on $S$ we consider the following sets
\[ 
\begin{split}
\mathcal P & :=\{ T\in \mathcal{D}'_{1,1}(S, \R) \, | \, T\ge 0 \}, \\
\mathcal K &:= \{ T\in \mathcal{D}'_{1,1}(S, \R) \, | \, T=d_{\beta} d^c_{\beta} \tau, \,   \tau \le 0,  \, \tau(\psi\wedge \psi)=-1\}. 
\end{split}
\]
$\mathcal P$ is a closed convex cone whereas $\mathcal K$ is a weakly compact (by the Banach--Alaoglu theorem) and convex subset of $\mathcal{D}'_{1,1}(S, \R)$. We want to show that $\mathcal P\cap \mathcal K \neq \emptyset$.

Suppose for contradiction that $\mathcal P \cap \mathcal K = \emptyset$. By the Hahn--Banach separation theorem \cite[p.65]{schaeffer}, there exists a smooth $(1,1)$-form $\tilde F$ such that  for any $T_1\in \mathcal P$,  $T_1(\tilde F)\ge 0$ and for any $T_2\in \mathcal K$, $T_2(\tilde F)<0$. The first condition is equivalent to  $\tilde F$ being semi-positive (i.e. $\tilde F\ge 0$) whereas the second condition is equivalent to $(d_{\beta}d^c_{\beta} \tilde F) >0$ on $S$. As the latter condition is open, we can assume without loss that $\tilde F>0$ is positive-definite.    Let $f$ be  a smooth function such that  $F:=e^{f} \tilde F$  is a Gauduchon metric, see \cite{gauduchon-0}. We now apply Theorem~\ref{thm:conformal}.  By the assumption and Proposition~\ref{p:lcs-characterization} there exist a real number $t\in (0,1)$ and a smooth function  $h$,  such that   $d_{\tilde \alpha} d^c_{\tilde \alpha}  F =0$ for the $1$-form $\tilde \alpha = \frac{1}{t}(\beta + dh)$. Letting $\beta' = \beta + df= \tilde \beta + d(f-h)$, it follows from \eqref{computation}
\[
\begin{split}
0<e^{f}\big(d_{\beta} d^c_{\beta} \tilde F\big) =&  e^{f} \Big(d_{\beta} d^c_{\beta} \big(e^{-f }F\big)\Big) = d_{\beta'}d^c_{\beta'} F \\
                                              =& \frac{1}{2}\Big(\delta^g (\beta' -\theta_g) + \langle \beta' - \theta_g, \beta'\rangle\Big) F\wedge F \\
                                              =&  \frac{1}{2}\Big(\Delta^g(f -h) +  \big\langle d(f-h), (2\tilde\beta-\theta_g)\big\rangle_g + |d(f-h)|_g^2 \Big)F\wedge F \\
                                              &+ d_{\tilde \beta}d^c_{\tilde \beta} F \\
                                              =&\frac{1}{2} \Big(\Delta^g(f -h) +  \big\langle d(f-h), (2\tilde\beta-\theta_g)\big\rangle_g + |d(f-h)|_g^2\\
                                               & \, \, \,  + \frac{(t-1)}{t} |\tilde \beta|_g^2 \Big)F\wedge F.
\end{split}
\]
It thus follows that everywhere on $S$, 
\[\Delta^g(f -h) +  \big\langle d(f-h), (2\tilde\beta-\theta_g)\big\rangle_g + |d(f-h)|_g^2 >0,\]
At a point where $f-h$ reaches its minimum $ \Delta^g(f -h)\le 0$  and $d(f-h)=0$, which is impossible.

\smallskip
We now consider the case when the minimal model of $S$ has zero second Betti number. By \cite{andrei}, $S$ is obtained from either a Hopf surface or an Inoue--Bombieri surface by blowing up points. In the first case $\tauc(S)=(-\infty, 0)$ by \cite[Theorem~1.4]{lcs}, so there is nothing to prove. In the former case, $\tauc(S)=\{ a\}$ by \cite[Theorem~1.3]{lcs1}. This  also corresponds  to the case when there exists an $\alpha \in a$ such that  $d_{-\alpha} d^c_{-\alpha} (\pm 1)=0$ (see \cite[Theorem~4.1]{lcs1}). We can argue in this case (see the proof of Proposition~\ref{bounds-examples} below) that  for any $b>a$ there exists  a $\beta \in b$ and a degree zero current $\tau <0$ such that $d_{\beta} d^c_{\beta} \tau >0$.

\smallskip                              
Part (b) follows by similar arguments. \end{proof}
                                     
                                     
    The previous result can be complemented with the following useful observation.
   
 \begin{prop} \label{p:sign} Suppose $b\in (-\infty, 0) \subset H^1_{dR}(S, \R)$  is such that there exists a closed $1$-form $\beta\in b$ and a degree $0$ current  $\tau_{\beta}$ on $S$  with $T:= d_{\beta} d^c_{\beta} \tau_{\beta} \ge 0, \, T\neq 0$.  Suppose, furthermore, that there exists $a\in \tauc(S)$ such that $a< b$. Then,  $\tau_{\beta}$  is a negative current.  \end{prop}     
 \begin{proof} 
 It follows by the assumptions on $T=d_{\beta} d^c_{\beta}\tau_\beta$ that
 \begin{equation}\label{assumption}
 (d_{\beta} d^c_{\beta}\tau_\beta) (F)>0
 \end{equation} 
 for any hermitian $2$-form $F$.  Letting 
 \[\tilde \beta:=\beta + d\varphi, \qquad \varphi\in C^{\infty}(S), \]
 and using that 
 \[ e^{\varphi} d_{\tilde \beta} d^c_{\tilde \beta} (e^{-\varphi} \tau_{\beta}) = d_{\beta} d^c_{\beta} \tau_{\beta}, \]
(which follows from the identity $e^{-\varphi} d_{\tilde \beta} d^c_{\tilde \beta} e^{\varphi} F = d_{\beta}d^c_{\beta} F$ on forms) we see that  $\tau_{\tilde \beta} := e^{-\varphi} \tau_{\beta}$ satisfies  the assumption \eqref{assumption} with respect to $\tilde \beta \in b$. Furthermore,   $\tau_{\tilde \beta}$ has the same sign as  $\tau_{\beta}$. This gives rise to a conformal class of degree $0$ currents  
\[ [\tau_{\beta}]=\{\tau_{\beta} \, , \, \beta\in b\}, \]
satisfying the assumption of the Proposition.

\smallskip
By \cite{lcs}, there exists a Gauduchon hermitian form $F_0$ and a closed $1$-form $\alpha_0\in a$  such that $d_{\alpha_0} d^c_{\alpha_0} F_0=0$.
Let $F= e^f F_0$ and  $\alpha = \alpha_0 + df$. We then have
\[  d_{\alpha} d^c_{\alpha} F = e^f(d_{\alpha_0} d^c_{\alpha_0} (e^{-f}F))= e^f(d_{\alpha_0} d^c_{\alpha_0} F_0)=0.\]
Our assumption $a<b<0$ yields that $b = t_0a,  \, t_0\in (0, 1)$.  Letting  $\beta := t_0\alpha$, we compute as in \eqref{computation} 
 \begin{equation*}\label{computation-1}
\begin{split}
d_{\beta}d^c_{\beta} F = & \Big(d(J\theta_g -J\beta) + (\theta_g-\beta)\wedge J(\theta_g -\beta)\Big)\wedge F \\
                                     =& \frac{1}{2}\Big(\delta^g(\beta-\theta_g) + \big\langle(\beta-\theta_g), \beta\big\rangle_g\Big) F\wedge F \\
                                     =&\frac{t_0}{2}\Big(\delta^g(\alpha -\theta_g) +\frac{(t_0-1)}{t_0}(\delta^g\theta_g) + \big\langle(t_0\alpha-\theta_g), \alpha\big\rangle_g\Big) F\wedge F  \\
                                     =&\frac{t_0}{2}\Big(\delta^g(\alpha -\theta_g) + \big\langle \alpha -\theta_g, \alpha\big \rangle_g +\frac{(t_0-1)}{t_0}(\delta^g\theta_g)  + (t_0-1)|\alpha|^2_g  \Big)F\wedge F \\
                                     =&\frac{(t_0-1)}{2}\Big(\delta^g\theta_g  + t_0 |\alpha|^2_g  \Big)F\wedge F,
                                      \end{split}
                                     \end{equation*}
                                     where we have used the equivalence (which follows from the equality at the second line with $\alpha = \beta$) 
                                     \[ \delta^g(\alpha -\theta_g) + \big\langle \alpha -\theta_g, \alpha\big \rangle_g=0 \, \Longleftrightarrow \, d_{\alpha} d^c_{\alpha} F =0,\]
                                     for passing to the last line.
                                     Using that the hermitian metric $g$ is related to the Gauduchon metric by $g=e^f g_0$, so that  $\theta_g=\theta_{g_0}+ df$,  and $\alpha= \alpha_0 + df$, we have (see e.g. \cite[pp. 59]{Besse})
                                     \[ \begin{split}
                                    \delta^g\theta_g &=e^{-f}\Big(\delta^{g_0} (\theta_{g_0} + df) - \langle df, \theta_{g_0} + df \rangle_{g_0}\Big) \\
                                                               &= e^{-f}\Big(\Delta_{g_0} f - \langle df, \theta_{g_0} + df \rangle_{g_0}\Big), \\
                                    |\alpha|^2_g &=e^{-f} \langle \alpha_0 + df, \alpha_0 + df\rangle_{g_0}. \end{split}\]
Hence
$$d_\beta d_\beta^c F=     \frac{(t_0-1)}{2}\Big( \Delta_{g_0} f  + \big\langle 2t_0\alpha_0 - \theta_{g_0}, df \big\rangle_{g_0} + (t_0-1)|df|^2_{g_0} +t_0|\alpha_0|^2_{g_0}\Big) e^f F_0\wedge F_0.$$
                                
Making the substitution $u := e^{(1-t_0)f}$, we conclude
\begin{equation}\label{computation2}
\begin{split}
d_{\beta}d^c_{\beta} F &=  -\frac{1}{2}\Big(\Delta_{g_0} u + \big\langle 2t_0\alpha_0-\theta_{g_0}, du \big\rangle_{g_0} + t_0(1-t_0)|\alpha_0|_{g_0}^2 u \Big) e^{t_0f}\, F_0 \wedge F_0.
\end{split}\end{equation}
 Consider, for $0<t_0<1$, the linear differential operator,
 \begin{equation*}\label{P}
 P(\varphi) :=     \Delta_{g_0} \varphi + \big\langle 2t_0\alpha_0-\theta_{g_0}, d\varphi \big\rangle_{g_0} + t_0(1-t_0)|\alpha_0|^2_{g_0}\varphi, \qquad \varphi \in C^{\infty}(S). \end{equation*}  
As $\alpha_0\neq 0$,  by the strong maximum principle (see e.g. \cite[p.349]{Evans})  $P$ has a trivial kernel. As the index of  $P$ is zero, $P$ has also a trivial co-kernel. It then follows from standard linear elliptic theory that for any smooth function $w\in C^{\infty}(S)$, there exists a unique smooth function $u\in C^{\infty}(S)$ solving the PDE
\begin{equation}\label{PDE}
 P(u)= w. \end{equation}
We can thus choose $w\ge 0, \, w\not\equiv 0$ and denote by $u$ the corresponding solution of \eqref{PDE}. Evoking    the strong maximum principle again, it follows that $u>0$ on $S$.  This shows that there exists a smooth function $f$ such that  $F:= e^{f}F_0$ satisfies
\[ d_{\beta} d^c_{\beta} F = -\frac{1}{2} w  e^{t_0f} \, F_0 \wedge F_0, \qquad  w\ge 0. \]
where $\beta := t_0(\alpha_0 + df)$ and $\beta_0 := t_0 \alpha_0$. Considering the corresponding currents $\tau_{\beta}$  and $\tau_{\beta_0}$ in the conformal class,    we have by \eqref{assumption}:
\[ 0< T(F) =d_{\beta} d^c_{\beta} \tau_{\beta}(F) = \tau_{\beta}\left(d_{\beta}d^c_{\beta} F\right)= -\frac{1}{2} \tau_{\beta} \left(e^{t_0f} w\, F_0 \wedge F_0\right) = -\frac{1}{2} \tau_{\beta_0} \left(w F_0\wedge F_0\right). \]
As this inequality holds for any $w\ge 0, w \not\equiv 0$, we conclude that $\tau_{\beta_0} <0$.     \end{proof}

We have the following corollary of the above results.

\begin{thm} \label{thm:tau-connected} Let $S$ be a class ${\rm VII}$ surface.  Then $\tauc(S)$ is connected.
\end{thm}
\begin{proof}  Recall that $\tauc(S) \neq \emptyset$ by \cite{lcs}.  If  $\tauc(S)$ is not connected,  we can find  $a_1, a_2\in \tauc(S)$ and $b\notin \tauc(S)$ with $a_1<b< a_2$. If the minimal model $S_0$ of $S$ has zero second Betti number, by \cite{andrei} $S$ is either a Hopf surface or an Inoue--Bombieri surface. In these cases the set $\tauc(S)$ is known and connected by the results in \cite{lcs,lcs1}. We now assume $b_2(S_0)>0$. 
By Propositions~\ref{p:lcs-characterization} and \ref{p:conformal-dual}, there exists a  zero degree current $\tau$ with $T:=d_{\beta} d^c_{\beta} \tau \ge 0, \beta \in b$ and $T\neq 0$. By Proposition~\ref{p:sign},  $\tau \le 0$ and by Proposition~\ref{p:obstruction}-(a), $a_2<b$, a contradiction.
\end{proof}

\begin{proof}[\bf Proof of Theorem~\ref{thm:tau-structure}] This in an immediate corollary of \cite[Theorem~1.1, Proposition~4.3]{lcs} (which show that $\tauc(S)$ is a non-empty subset of $(-\infty, 0)$), \cite[Theorem~1.3]{lcs1} (which shows that $\tauc(S)$ is either open or a single point and characterizes the latter case) and Theorem~\ref{thm:tau-connected}. \end{proof}
 
 \begin{defn}\label{d:cyclic}   The {\it minimal $\Z$-cover} of a complex surface $S$ with $b_1(S)=1$ is defined as follows.  The torsion-free part $H_1(S, \Z)^{\rm free}$ of  $H_1(S, \Z)$ being $\Z$, $S$ admits an infinite cyclic cover $\hat \pi:  \hat S    \to S$  whose  fundamental group is the kernel of the morphism    
\[ \pi_1(S) \mapsto     \Big(\frac{\pi_1(S)}{[\pi_1(S), \pi_1(S)]}\Big)^{\rm free}.\]
By construction,  $\hat \pi^*\Big(H^1_{dR}(S, \R)\Big)= \{0\}$ and $\hat S$ can be characterized as the smallest $\Z$-cover of $S$ with that property.  
\end{defn}

We denote by $\gamma$ the deck transformation on $\hat S$ such that $S= \hat S/\langle \gamma \rangle$ and notice  the following simple fact.
\begin{lem}\label{l:cover} Let $S$ be a compact complex surface with $b_1(S)=1$ and $\hat S$ be its minimal $\Z$ cover with $S = \hat S/\langle \gamma \rangle$. Then $S$ admits a closed  $1$-form $\beta$ and a weakly negative (resp. weakly positive) current $\tau$ of degree $0$  such that $T:= d_{\beta}d^c_{\beta} \tau \ge 0$  if and only if $\hat S$ admits a negative (resp. non-negative) locally integrable plurisubharmonic function $\hat u$,  which is automorphic in the sense that
\begin{equation}\label{automorphy}
\hat u \circ \gamma = C \hat u
\end{equation}
for a positive constant $C$.
\end{lem}   
\begin{proof} Suppose first $\beta$ is a closed $1$-form  and $\tau$ a  degree $0$ current on $S$ satisfying $T:= d_{\beta}d^c_{\beta} \tau \ge 0$. We can think of $\tau$ as a function in $L^1_{loc}(S)$ and 
denote by $\hat \tau$ the lifted function to $\hat S$. This gives rise to a degree $0$ current on $\hat S$  such that $\gamma^*(\hat \tau)=\hat \tau$. Let $\hat \beta = \hat\pi^*(\beta)$ be  the lift of the closed $1$-form $\beta$. We then define $\hat T := d_{\hat \beta}d^c_{\hat \beta} \hat \tau$. It is not hard to see that $\hat T$  is a positive current on $\hat S$ (as $T=d_{\beta}d^c_{\beta} \tau \ge0$ on $S$). We can write  $\hat \beta =  d\hat f$ on $\hat S$ for a smooth function satisfying 
\[ \hat f \circ \gamma - \hat f = c\]
for some real constant $c$. 
It then follows that the degree $0$  current $e^{\hat f} \hat \tau$  on $\hat S$  satisfies 
\[ dd^c  \big(e^{\hat f} \hat \tau\big)= e^{\hat f} \big(d_{\hat \beta} d^c_{\hat \beta} \hat \tau\big)  = e^{\hat f}\hat T \ge 0\]
It follows that (see e.g. \cite[Ch.I, (5.5)]{demailly-book} ) $e^{\hat f} \hat \tau$ defines a negative (resp. non-positive) PSH function $\hat u \in L^1_{loc}(\hat S)$ (where we have used the maximum principle to rule out zeroes of $\hat u$ in the non-positive case) which has the property $\hat u \circ \gamma = e^c \hat u$.

\smallskip
Conversely, suppose $\hat u$ is a negative (resp. non-negative) PSH function on $\hat S$ satisfying $\hat u \circ \gamma = e^c \hat u$ for some real constant $c$. As $H^1_{dR}(S, \R) \neq \{0\}$,  there exists a closed but not exact $1$-form $\beta$  whose lift on $\hat S$ satisfies $\hat \beta=d\hat f $ for a smooth function $\hat f$ such that $\hat f \circ \gamma =\hat f+c_0$, $c_0\neq 0$. Multiplying $\hat f$  by a suitable constant if necessary, we may suppose that $c_0=c$.  We then consider the degree zero current  $\hat \tau := e^{-\hat f}\hat u$ on $\hat S$ 
 and check as before that
\[ d_{\hat \beta} d^c_{\hat \beta} \hat \tau = e^{-\hat f} dd^c \hat u \ge 0,\]
where $\hat \beta = d\hat f$.   Furthermore, as $e^{-\hat f}\hat u$ is $\gamma$-invariant and negative (resp. non-negative), there exists a degree zero weakly negative (resp. weakly positive) current $\tau$ on $S$ such that $\hat \tau$ is the pull-back of $\tau$ to $\hat S$.   Using that $\gamma^*(\hat \beta)=\hat \beta$  we conclude that $\tau$ satisfies
$d_{\beta} d^c_{\beta} \tau  \ge 0$ on $S$ with respect to the induced $1$-form $\beta$ from $\hat \beta$. \end{proof}   

\begin{rem}\label{r:convention} For any closed $1$-form $\alpha$ on $S$, the lift $\hat \alpha$ to $\hat S$ is exact, and  $\hat \alpha = d\hat f$ for a smooth function $\hat f$ satisfying 
\[ \hat f \circ \gamma - \hat f = c\]
for some real constant $c$.  The constant $c=c_a$ is independent of the choice of $\alpha$ in a given deRham class $a\in H^1_{dR}(S, \R)$,  being zero precisely when $\alpha$ is exact.  It follows that the positive constant $C_a = e^{c_a}$  is  $1$ iff $a=0$.

There is a $2$-fold choice for the generator $\gamma$: we  can  replace $\gamma$ by  $\gamma^{-1}$, which  changes $C_a$ with $1/C_a$. To fix this ambiguity, in this paper we  assume that the generator $\gamma$ is choses so as $C_a>1$ iff $a<0$. As a corollary of Proposition~\ref{p:obstruction}(a),  the constant $C=C_b$ in Lemma~\ref{l:cover} must satisfy $C>1$ (this is equivalent to the inequality $b<0$).

Furthermore, using that the function $f(x)= -(-x)^t$ is increasing and convex on $(-\infty, 0)$ for any real  number $t\in (0,1)$, we conclude that $-\big(-\hat u\big)^{t}$ is also a negative PSH function on $\hat S$ which satisfies \eqref{automorphy} with constant $C^t <C$; it thus follows that if $b$ satisfies the hypothesis of Proposition~\ref{p:obstruction}(a), then any $b' \in (b, 0)$ does too. 
\end{rem}

\begin{rem}\label{convention1}  Related to Remark~\ref{r:convention} is the following useful observation. On any compact complex surfaces satisfying $b_1(S)=1$ and ${\rm kod}(S)=-\infty$  we have the identification (see \cite{kodaira})
\begin{equation}\label{Pic-identification}
H^1_0(S, \C^{*}) = {\rm Pic}^0(S) \cong \C^*,
\end{equation}
compare with \eqref{H1}. 
We can then write for an isomorphism class of flat bundles $\cL\in  H^1_0(S, \C^{*})$ 
\[ \cL=\cL_\mu, \qquad \mu \in \C^*.\]
 In terms of the sequence \eqref{H1}, for any $a\in H^1_{dR}(S, \R)$, the corresponding flat complex bundle $\cL_a = L_{a}\otimes \C$ can be then written as $\cL_{C_a}$, where $C_a>0$ is the constant introduced according to the convention in Remark~\ref{r:convention}. We thus have that $\cL_{\mu} = \cL_{a}$ with $a\in (-\infty, 0)$ iff $\mu \in \R, \mu>1$, an identification that we will use throughout the paper.
\end{rem}

\section{Examples}\label{s:examples} In this section, we examine the existence of degree $0$ currents as in Proposition~\ref{p:obstruction} on the known surfaces in the Kodaira class {\rm VII} (see \cite{bpv}), i.e. the compact complex surfaces satisfying $b_1(S)=1$ and ${\rm kod}(S)=-\infty$.  Notice that the existence of such currents is invariant under blow-up and blow-down of points, so we shall restrict attention to the minimal case.  Minimal complex surfaces in the class VII are  referred to as \emph{class {\rm VII}$_0$ complex surfaces}.

We  first recall the  rough classification of the known minimal complex surfaces of Kodaira class VII.

\subsection{Known examples of class VII$_0$ complex surfaces}
\begin{defn}\label{VII} A {\it known} minimal complex surface in the Kodaira class {\rm VII}$_0$  is one of the following
\begin{enumerate}
\item[$\bullet$] An Inoue--Bombieri surface~\cite{I74} ($b_2(S)=0$ and $S$ admits no curves);
\item[$\bullet$] A Hopf surface~\cite[II, section 10]{kodaira} and \cite{kato-hopf} ($b_2(S)=0$ and $S$ admits curves);
\item[$\bullet$] A Kato surface~\cite{K77,D84}($b_2(S)>0$ and $S$ admits a global spherical shell).
\end{enumerate}
Following \cite{D84,Brunella14}, we shall further distinguish between two different types of Kato surfaces
\begin{enumerate}
\item[$\bullet$] Enoki surfaces~\cite{enoki}, and
\item[$\bullet$] hyperbolic Kato surfaces~\cite{Brunella14}.
\end{enumerate}
\end{defn}
If $b_2(S)=0$ there are no other surfaces than the Inoue-Bombieri and Hopf surfaces listed above~\cite{andrei}. The   minimal surfaces in the class VII are said to belong in the class ${\rm VII}_0$ whereas those with $b_2(S)>0$  are said to belong to the class VII$_0^+$. 
According to the still open  GSS conjecture (proved in \cite{andrei1} when $b_2(S)=1$) any complex surface in the class VII$_0^+$ must be a Kato surface.

\subsection{Definite  currents on the known class ${\rm VII}_0$ surfaces}
We prove with the following 
\begin{prop}\label{bounds-examples} Suppose $S$ is  a known minimal surface of the Kodaira class {\rm VII}$_0$. Then, 
\begin{enumerate}
\item[\rm (a)] $S$ admits a weakly negative current $\tau \le 0$ with $T=d_{\beta}d^c_{\beta} \tau \ge 0$ and $T\neq 0$  for some closed $1$-form $\beta$ if and only if  $S$ is either an Inoue--Bombieri surface, or a hyperbolic Kato surface;
\item[\rm (b)]  $S$ admits a weakly positive current $\tau\ge 0$ with $T=d_{\delta} d^c_{\delta} \tau \ge 0$ and $T\neq 0$ for some $\delta \in d\in (-\infty, 0)\subset H^1_{dR}(S, \R)$  if and only if $S$ is an Inoue--Bombieri surface.
\end{enumerate}
\end{prop}
\begin{proof} We start with showing nexessary part of (a): to this end, we need to prove that Hopf and Enoki surfaces do not admit a degree $0$ current $\tau$ as in the case (a) of the Proposition.  If  a degree zero current $\tau \le 0$ existed on a Hopf surface $S$, it would define a non-positive PSH function  $\hat u$ on the universal cover $\C^2\setminus\{0\}$ by Lemma~\ref{l:cover}. As $\C^* \times \C \subset \C^2\setminus\{0\}$ is covered by $\C^2$, by Liouville theorem we must have that $\hat u$ is constant, and thus $T=0$, a contradiction.  The case of Enoki surfaces can be treated similarly: we notice that such a surface contains a cycle of rational curves whose complement has a covering  biholomorphic to $\C^*\times \C$ (see~\cite{enoki}, \cite[Corollary~1.20]{DK98}). 


\smallskip
We now show (b).  To this end, we need to check that a Hopf or a Kato surface do not admit a degree $0$ current  as in the statement. We shall argue by absurd, using Proposition~\ref{p:obstruction}~(b). In the case of a Hopf surface $S$, we have by \cite[Theorem~1.4]{lcs}  that $\tauc(S)=(-\infty, 0)$ so we conclude by  Proposition~\ref{p:obstruction}~(b). Similarly, if $S$ is a Kato surface,  Brunella~\cite{B1,B2} shows that  $\tauc(S)$ contains an interval $(-\infty, c)$ for some $c<0$, so we conclude again by Proposition~\ref{p:obstruction}~(b).

\smallskip
We finally  address the existence part. Notice that Inoue--Bombieri surfaces admit a closed $1$-form $\alpha$ with $d_{-\alpha}d^{c}_{-\alpha} (1)=0$  and $a:=[\alpha] \in (-\infty, 0) \subset H^1_{dR}(S, \R)$ (see e.g. \cite{lcs1}). It follows that the integration over $S$ is a positive degree $0$ current on $S$ which is $d_{\alpha}d^c_{\alpha}$-closed. By the proof of Lemma~\ref{l:cover}, it gives rise to  a strictly positive PH function $\hat u=e^{-f}$, where $\hat\beta=df$,  on the minimal $\Z$-cover $\hat S$, which is automorphic with constant $C_0>1$. Taking $(\hat u)^t, t>1$ and $-(\hat u)^{t}, t\in (0,1)$ we see that for any $C_0<C$ there exists a positive strictly PSH automorphic function on $\hat S$ with automorphy constant $C$, and for any $1<C'<C_0$,  there exists a negative automorphic strictly PSH function on $\hat S$ with automorphy constant $C'$. Equivalently, for any $\delta \in d<a=[\alpha]$, there exists a weakly positive degree $0$ current on $S$ satisfying  the hypothesis of (b), and for any $\beta \in b>a$, there exists a current  of degree $0$ on $S$ satisfying the hypothesis of (a). Proposition~\ref{p:obstruction} is then consistent with the fact that in this case $\tauc(S)=\{a\}$, see \cite{lcs1}. The existence of negative PSH functions on the $\Z$-cover of a hyperbolic Kato surface is treated in detail in the next subsection. \end{proof}

\begin{proof}[\bf Proof of Theorem~\ref{main}]  $\tauc(S)$ is a non-empty  subset of $(-\infty, 0)$ by \cite[Theorem~1.1, Proposition~4.3]{lcs}, see  also Proposition~\ref{p:tau} above. Part (a) is established in \cite[Theorem~1.3]{lcs1};  for part (b) see e.g. \cite[Theorem~1.4]{lcs} and \cite{tsukada}. In the cases (c) and (d), $\tauc(S)$ is open by \cite[Theorem~1.3]{lcs1}  and connected by Theorem~\ref{thm:tau-connected}. The lower bound follows from  \cite{B1,B2} and the upper bounds in each case follow from Proposition~\ref{bounds-examples} above. \end{proof}

\begin{proof}[\bf Proof of Corollary~\ref{c:Cheeger-Gromov}] Let $\tilde J_i : = \phi_i(J)$ be the integrable almost complex structure on $M_k$ obtained by the pull-back with  a  diffeomorphism $\phi_i \in {\rm Diff}_0(S)$, and denote by $S_i:=(M_k, \tilde J_i)\cong S$ the corresponding complex surface.  As $\phi_i$ acts trivially on $H^1_{dR}(M_k, \R)$, we have by Theorem~\ref{main} $\tauc(S_i) = \tauc(S) = (-\infty, b) \subset H_{dR}^1(M_k, \R)$. By assumption, $\tilde J_i$ converges in $C^{\infty}$ to a complex structure $J_{\infty}$: it then follows that any LCS structure $\omega$ taming $J_{\infty}$ will also tame $\tilde J_i$ for $i$ large enough. This  and  $\tauc(S_i) = \tauc(S)$ show that $\tau(S_{\infty}) \subset \tauc(S)= (-\infty, b)$. Applying Theorem~\ref{main} again, we conclude that $S_{\infty}$ must be a hyperbolic Kato surface. \end{proof}

\subsection{An explicit upper bound of the Lee classes of a hyperbolic Kato surface}\label{s:explicit-bound} In this section, we show the existence of  an upper bound $b=b(S)<0$ of $\tauc(S)$ given by Proposition~\ref{p:obstruction} for a hyperbolic Kato surface $S$. Recall that a Kato surface $S$ is   a minimal compact complex surface in the class VII with $b_2(S)>0$ and containing a GSS. Such a surface is obtained by the following general construction. Given
 \begin{itemize}
 \item a composition $\pi=\pi_0\circ\cdots\circ\pi_{n-1}:W\to \bb C^2$ of $n=b_2(S)$ blow-ups over the origin of $\bb C^2$, in such a way that $\pi_0$ blows up the origin of $\bb C^2$, $C_0=\pi_0^{-1}(0)$ and $\pi_i$ blows up a point $O_{i-1}\in C_{i-1}$ for any $i\ge 1$, and
 \item a biholomorphism $\s:\bar B\to W$ from a neighbourhood of the closed ball to $W$ such that $\s(\bar B)\subset \pi^{-1}(B)$ and $\s(0)\in C_{n-1}$,
 \end{itemize}
 we remove $\s(B)$ from $\pi^{-1}(\bar B)$.  We obtain an annulus with a boundary composed of two spheres $\pi^{-1}(S^3)$ and $\s(S^3)$ where  $S^3=\part B$. The surface $S=S(\pi,\s)$ is obtained by glueing by $\s\circ\pi$  isomorphic neighbourhoods of these two spheres (see \cite{D84} for details). The resulting smooth manifold is diffeomorphic to $(S^1\times S^3)\sharp k{\overline {\C P^2}}$ and the minimal $\Z$-cover $\hat S$ of $S$ (see Definition~\ref{d:cyclic}) coincides with the universal cover.  Let
 $$F=\pi\circ\s:(\bb C^2,0)\to (\bb C^2,0)$$
 be the associated germ. The order of creation of  rational curves induced by the successive blow-ups produces a total order on the rational curves in $\hat S$.  Given one of these curves, say $C$, it is possible to collapse all the curves $C'>C$ to a point $\hat O_C$ of $C$. Doing this, $C$ becomes an exceptional curve of the first kind and we denote by $\hat S_C$ the resulting manifold and by $p_C:\hat S\to \hat S_C$ this collapsing. Notice that $p_C$ is not proper, and induces by restriction an isomorphism 
$$p_C:\hat S\setminus \bigcup_{C'>C} C'\to \hat S_C\setminus\{\hat O_C\}.$$
We shall use next the following fact taken from \cite[Proposition 3.9]{D84}. Let  $\gamma: \hat  S\to \hat S$ be the generator of $\pi_1(S)\cong \Z$.  Then,  for any curve $C$ in $\hat S$ as above, we have the following commutative diagram
\begin{diagram}
& {\hat S} & \rTo^{\gamma}&  {\hat S} \\
& \dTo^{p_C} &                 & \dTo_{p_C} \\
&{\hat S}_C & \rTo^{F_{C}} & {\hat S}_C
\end{diagram}
Moreover, if  $\hat O_C$ is the critical value of $p_C$ (i.e. the image of all the rational curves $C'>C$ with notations of \cite{D84}),  we  can identify the germ $(\hat S_C,\hat O_C)$ with $(\bb C^2,0)$ in such a way that the equation of $C$ at $\hat O_C$ is $C=\{z_2=0\}$.
\begin{lem} \label{pshZrev} Let $S$ be a  Kato  surface defined by a germ 
$$F=\pi\circ\s:(\bb C^2,0)\to (\bb C^2,0)$$
 and $v_C: (\hat S_C,\hat O_C)\to \R \cup \{-\infty\}$  a germ of PSH function such that $F_C^* (v_C)=v_C+c$ for a constant $c\in\R$. Then,  there exists on $\hat S$ a unique PSH function $\hat v :\hat  S\to \R \cup \{-\infty\}$ satisfying $\gamma^* \hat v =\hat v + c$, and such that  if $v_C$ is defined in a neighbourhood $V$ of $\hat O_C$,  then on $p_C^{-1}(V)$ we have 
$$v_C\circ p_C={\hat v}.$$
\end{lem}
\begin{proof}The inverse image $p_C^{-1}(V)$ contains the pseudo-concave end of  the universal covering space $\hat S$ (in the terminology of \cite{D84}).  Notice that $\hat  S=\bigcup_{i\in\bb Z}A_i$ where $A_i$ is a copy of a fundamental domain isomorphic to an annulus $\overline{\Pi^{-1}(B)}\setminus \s(B)$. We have $\gamma (A_{i-1})=A_i$ for every $i\in\bb Z$. If $A_i\subset p_C^{-1}(V)$  we define $\hat v$ on $A_{i-1}$ by $\hat v(x):=\hat v(\gamma(x))-c$. \end{proof}

We now recall that the hyperbolic Kato surfaces are either Kato surfaces of intermediate type or  Inoue--Hirzebruch surfaces which we will treat in the next two examples, respectively.

\begin{example}~\label{ex:Kato-intermediate} Let $S$ be a Kato surface of intermediate type, i.e. a Kato surface  which contains a cycle of rational curves with branches. A defining germ of $S$ is  given by (see \cite{Fav00})
\begin{equation}\label{germ-fabre}
F(z_1,z_2)= (\l_1 z_1z_2^s+P(z_2)+\lambda_2z_2^{\frac{sk}{k-1}},z_2^k), 
\end{equation}
where  $\l_1\in \C^*,  \lambda_2\in \C, s, k\in \N^*, k\ge 2$,  $\frac{s}{(k-1)}\in \N$ when $\lambda_2\neq 0$,   and  $P(z)$ is a certain polynomial  with complex coefficients and  $\deg(P)\le s$. The maximal divisor $D$ of $S$ is defined by the equation $z_2=0$.
For $|z_2]<1$,  the function $$ v(z_1,z_2):=-\log(-\log|z_2])$$
 is PSH,  and satisfies 
 $$F^* v= v-\log(k).$$
By Lemma~\ref{pshZrev}, $v$ gives rise to a  PSH  function $\hat v$ on $\hat S$,  satisfying $\gamma^* \hat v = \hat v -\log(k)$. Define $\hat u := -\exp(-\hat v) = \log|z_2| < 0$. It follows that $\hat u$ is negative PSH satisfying
\[\gamma^* \hat u = k \hat u.\]
The fact that $k>1$ tels us that the choice for the generator is consistent with the  convention of Remark~\ref{r:convention}.
Furthermore, the support of $\hat u$ is on $D=\{z_2=0\}$. It follows that in this case, the upper bound $b\in (-\infty, 0)$ of Proposition~\ref{p:obstruction} is given by the deRham class corresponding to the topologically trivial flat line bundle defined by $k\in \N^*$.
\end{example}
\begin{example}\label{ex:Inoue-Hirzebruch} Let $S$ be an Inoue--Hirzebruch  Kato surface. In this case, the contraction germ defining $S$ is given by  (see  \cite[p.667]{Dl88})
 $$F(z_1,z_2)=(z_1^pz_2^q,z_1^rz_2^s)$$ 
 where  $A=\left(\begin{array}{cc}p&r\\q&s\end{array}\right) \in {\bf GL}(2, \Z)$  has  non-negative coefficients and $\det A=\pm 1$. Furthermore, the characteristic polynomial
$$P_A(X)=X^2-(p+s)X+\det A$$
has two eigenvalues $\l_1$, $\l_2$ satisfying $0<|\l_2|<1<\l_1$, $\l_2=\pm \l_1^{-1}$. Each eigenvalue $\l_i$ ($i=1,2$) has an eigenvector $(\xi_i,\eta_i)$ with $\xi_1>0, \eta_1>0$  and $\xi_2>0, \eta_2<0$. We next  let $(\xi, \eta):=(\xi_1, \eta_1)$,  $\lambda := \lambda_1>1$,  and 
$$G(z_1,z_2):=\xi\log|z_1|+\eta\log|z_2|.$$
We have that $G\circ F(z)=\l G(z)$ and  $\Omega:=\{(z_1, z_2) \in \C^2 : G(z_1, z_2)<0\}$  contains a neighbourhood of $0\in\bb C^2$. We then define  (on $\Omega$)  the PSH function $v(z):=-\log(-G(z))$ which satisfies $v\circ F =v +\log \l$ and hence, by  Lemma~\ref{pshZrev}, gives rise to a PSH function $\hat v$ on $\hat S$ satisfying $\gamma^*\hat v = \hat v + \log \l$. Letting $\hat u := -\exp(-\hat v)=G(z)$ we get a negative PSH function on $\hat S$ satisfying 
$$\gamma^* \hat u = \lambda \hat u.$$
It follows in this case too that  the upper bound $b\in (-\infty, 0)$ of Proposition~\ref{p:obstruction}  corresponds to the  constant $C=C_b=\lambda >1$, see Remark~\ref{r:convention}.
\end{example}

\subsection{An obstruction to the existence of bi-hermitian structures  on intermediate Kato surfaces} By  a {\it bi-hermitian structure}  on a connected compact complex surface $S=(M, J)$ we mean a conformal class $\mathsf{c}$ of hermitian metrics  such that there exists another integrable complex structure $I$ on $M$ which is positively oriented, $\mathsf{c}$-orthogonal, and $I(x) \neq \pm J(x)$ at at least one point of $S$. The main interest in bi-hermitian structures comes from their link with generalized K\"ahler geometry revealed in \cite{gualtieri}. When $b_1(S)$ is even, i.e. on  a K\"ahler complex surface,  it follows by the results in \cite{AGG,Goto} the existence of a bi-hermitian structure  is equivalent to the condition $H^0(S, K^{*}_S) \neq 0$ where $K_S$ stands for the canonical bundle and $K_{S}^*$ denotes its dual. On the other hand, in the case $b_1(S)$ odd, the situation is not completely understood yet.  It is known from~\cite{apostolov}  that  in this case $b_1(S)=1,  {\rm kod}(S)=-\infty$,  and  $H^0(S, K^*_{S}\otimes {\mathcal L}_{a}) \neq 0$ for some flat holomorphic line bundle $\mathcal{L}_{a}$ corresponding to a class  $a\in (-\infty, 0] \subset H^1_{dR}(S, \R)$ via \eqref{H1}. It was observed  in \cite{Dl,AD} that  the latter conditions imply that $S$ must be obtained by blowing up points of either a Hopf surface or a Kato surface of index 1 (or NAC  of index $1$ in the terminology of \cite{Dl}).  Since the existence of bi-hermitian structures is stable under blow-ups preserving  the  index $1$ NAC condition~\cite{CG}, one can concentrate on the minimal case, and one situation in which the existence of bi-hermitian structures remains unknown~\cite{FP19} is the case when $S$ is a Kato surface of intermediate type, as described in Example~\ref{ex:Kato-intermediate} above. In this case,  the necessary index 1  NAC condition is studied in detail in \cite{Dl,FP15,P19}. It is observed there that if  $H^0(S, K^*_{S}\otimes {\mathcal L}) \neq 0$,  then $\mathcal L = \mathcal L_{\mu}$  for a unique  $\mu\in \C^*$, where we have used the identification of $\cL$ with a $\mu \in \C^*$  as in Remark~\ref{convention1}.   In terms of the germ \eqref{germ-fabre},  the index $1$ NAC condition is equivalent to $\frac{s}{(k-1)} \in \N$ (see \cite{Fav00}). Indeed,  in this  case 
\[\sigma:= z_2^r \big(\frac{\partial}{\partial z_1}\wedge \frac{\partial}{\partial z_2}\Big), \, r= \frac{s}{(k-1)} +1, \]
defines a germ  satisfying $F_*(\sigma) = (\lambda_1k) \sigma$, so it defines a holomorphic  section of $K^{*}_S \otimes {\mathcal L}_{\mu}$ with $\mu = \lambda_1k$. In particular,  as noticed in \cite{P19}, $\mu\in \R$  and $\mu>1$ are necessary conditions for an intermediate Kato surface $S$ to admit a bi-hermitian structure.  We now prove that they are not in general sufficient.
\begin{thm}\label{thm:BH} Suppose $S$ is an intermediate Kato surface as described in Example~\ref{ex:Kato-intermediate}, which admits a bi-hermitian structure. Then $\lambda_1$ in \eqref{germ-fabre} must be a real number satisfying $\lambda_1> 1$ or, equivalently,   $H^0(S, K^{*}_S\otimes {\mathcal L}_{\mu})\neq 0$ where $\mu$ is a positive real number satisfying $\mu>k$.
\end{thm}
\begin{proof} By \cite{Dl}, ${\mathcal L}_{\mu}$ is the only flat line bundle on $S$ such that $H^0(S, K^*_{S}\otimes {\mathcal L}_{\mu}) \neq 0$ and, furthermore,  in this case  $H^0(S, K^*_{S}\otimes {\mathcal L}_{\mu}) =\C$ and the support of the NAC divisor is the maximal curve of $S$ (which is connected). It follows from \cite[Proposition~2.4]{ABD} that any bi-hermitian structure on $S$ (if it exists) must be given by the construction of \cite[Proposition~2.3]{ABD}. In particular, $J$ must be tamed by a locally conformally symplectic $2$-form $\omega$ with Lee deRham class $a$ and ${\mathcal L}_{\mu}= \mathcal{L}_{a}$ as explained above. The claim follows from Proposition~\ref{p:obstruction} and  Example~\ref{ex:Kato-intermediate}. \end{proof}

\section{Characterization of some known class VII surfaces}\label{s:conjecture}

Lemma~\ref{l:cover} allows us to recast  the results of M. Brunella~\cite{Brunella13,Brunella14} which  in turn provide a partial converse of Proposition~\ref{bounds-examples}(a).
\begin{thm}\label{thm:brunella}\cite{Brunella13, Brunella14} Let $S$ be a compact complex surface with $b_1(S)=1$ and ${\rm kod}(S)=-\infty$. Then the following conditions are equivalent
\begin{enumerate}
\item[\rm (i)] $S$ is obtained from either a hyperbolic Kato surface or a Bombieri--Inoue surface by blowing up points.
\item[\rm (ii)] $S$ admits a weakly negative non zero current $\tau \le 0$  of degree $0$, and a smooth closed $1$-form $\beta$ such that  
\[ T:=d_{\beta}d^c_{\beta} \tau \ge 0  \]
and  the support of $T$ is either empty or is an analytic set $Z$ of pure dimension $1$ in $S$.
\end{enumerate}
\end{thm}
\begin{proof} By Lemma~\ref{l:cover}, we can replace condition (ii) with the existence of non-positive PSH automorphic  function on $\hat S$.

The  direction ${\rm (i)} \Longrightarrow {\rm (ii)}$ can be  established explicitly, see Section~\ref{s:examples}  for the case of a hyperbolic Kato surface. 

To establish the direction ${\rm (ii)} \Longrightarrow {\rm (i)}$, we first notice that the support of $T$ is empty iff $T=0$ (and hence $\hat T=0$). In this case,  the PSH function $\hat u$ on $\hat S$ is pluriharmonic, and therefore  is smooth and strictly negative. In particular,  the constant $C$ in Lemma~\ref{l:cover} cannot be $1$, i.e. $b:=[\beta]\neq 0$ in $H^1_{dR}(S, \R)$. The equality in \eqref{inequality} (and using that $\tau$ is  a negative smooth function) now gives $\tauc(S)=\{b\}$. It follows that  $S$ cannot be a blow up of a Hopf surface (in which case $\tauc(S)=(-\infty, 0)$, see e.g. \cite{lcs}), so the algebraic dimension of $S$ is zero (see e.g. \cite{bpv}). Thus, 
$S$ must be a blow up of a Bombieri--Inoue surface by \cite{Brunella13}.

 In the case when $T$ is supported over a curve, $T\neq 0$, and by  Proposition~\ref{p:obstruction} we have $\tauc(S) \in (-\infty, b)$. Again, $S$ cannot be a blow-up of a Hopf surface, and thus the algebraic dimension of $S$ is zero. As  $\hat S$ admits a  negative PSH function  with analytic support over a curve,  $S$ must be a blow up of a hyperbolic Kato surface by \cite{Brunella14}. \end{proof}

 The following result was obtained in \cite{lcs1}.
 \begin{thm}\label{thm:ib}\cite{lcs1} Let $S$ be a compact complex surface  with $b_1(S)=1$ and ${\rm kod}(S)=-\infty$. Then the following conditions are equivalent
 \begin{enumerate}
 \item[\rm (i)] $\tauc(S)=\{a\}$ consists of a single point;
 \item[\rm (ii)] $\tauc(S)$ is a proper closed subset of $(-\infty, 0)$;
 \item[\rm (iii)]  $S$ is obtained by  blowing up points of an Inoue--Bombieri surface.
 \end{enumerate}
 \end{thm}

\section{Lee classes and twisted logarithmic holomorphic $1$-forms.}\label{s:logarithmic}
\subsection{Twisted logarithmic $1$-forms}
\begin{defn}\label{d:logarithmic} Let $X$ be a complex manifold, $D\subset X$ an effective divisor, and $\omlog$ a holomorphic $p$-form on $X\setminus D$. We say that $\omlog$ is a \emph{logarithmic $p$-form} on $X$ with logarithmic pole along $D$ if  $\omlog$ and $d\omlog$ have a pole of order at most one along $D$. 
We denote by $\Omega^p(\log D)$ the coherent sheaf of logarithmic $p$-forms on $X$ with pole along $D$. We shall  also consider logarithmic $p$-forms with values in a holomorphic line bundle $\cal L$, i.e. sections  of $H^0(X, \Omega^p(\log D)\otimes \cal L)$. In particular, an $\cal L$-valued logarithmic $1$-form $\omlog\in H^0(X, \Omega^1(\log D)\otimes \cal L)$ is a $\bar\partial_{\cal L}$-closed $\cal L$-valued $(1,0)$-form on $X\setminus D$ such that, around each point $p\in D$, there exists a neighbourhood $U$, holomorphic functions $f_j, g_j \in \cal O(U), j=1, \ldots, k$, a holomorphic $1$-form $\omlog_0\in \Omega^1(U)$, and a holomorphic section $s\in H^0(U, \cal L_{|_U})$ with
\begin{equation}\label{log-1-form}
\omlog_{|_U} = \Big(\omlog_0 + \sum_{j=1}^k g_j \frac{df_j}{f_j}\Big)\otimes s, 
\end{equation}
where  $\{f_j=0\}$ are the  reduced equations of the irreducible components of $D\cap U$.
\end{defn}
We shall be interested in particular in \emph{twisted} logarithmic $1$-forms on a complex surface $S$, i.e.  elements $\omlog\in H^0(S, \Omega^1(\log D)\otimes \cal L)$ where $\cal L \in H^1_0(S, \C^*)$ is a  \emph{topologically trivial} flat  $\C^*$-line bundle.   Using the short exact sequence
\begin{align*}
 \{0\} \to \Z \hookrightarrow \bb C \stackrel{\exp 2\pi i }{\longrightarrow} \bb C^* \to \{1\}, 
 \end{align*}
it follows that the first morphism in 
\begin{equation}\label{H1C}
H^1_{dR}(M,\bb C)\stackrel{\exp}{\longrightarrow}H^1_0(M,\bb C^*) \longrightarrow {\rm Pic}^0(X), 
\end{equation}
is surjective, i.e.  $\cal L$ can be written as $\cal L= \cal L_{\alpha}$ for a closed \emph{complex-valued} $1$-form $\alpha$.  As explained in Section~\ref{s:twisted-differentials},  $\omlog$ can be equivalently viewed as a $\bar\partial_{-\alpha}$-closed $(1,0)$-form on $S\setminus D$,  such that we can choose a neighbourhood $U$ of $p\in D$   with   $\alpha_{|_U}=dh$ for a smooth complex valued function $h$,   holomorphic $1$-form $\omlog_0$ and holomorphic functions $f_j, g_j$ on $U$ with
\begin{equation}\label{twisted-log}
e^{h}\omlog_{|_U} = \omlog_0 + \sum_{j=1}^k g_j \frac{df_j}{f_j}\Big.
\end{equation}
It follows from the above expression that $\omlog$ gives rise to a $(1, 2)$-current  $P_\omlog$,  defined by 
\begin{equation}\label{omega-current}
P_{\omlog}(\phi) := \int_{S\setminus D} \omlog \wedge \phi, \, \qquad  \forall \phi\in \cal{E}^{1,2}(S, \C).
\end{equation}
Furthermore,  assuming that $D$ is a simple normal crossing divisor, we compute  (using  \eqref{twisted-log} and the basic identity $i \partial \bar \partial \log |z|^2 = 4\pi \delta_0$)
\begin{equation}\label{log-current}
T_{\omlog}:= \bar \partial_{\alpha} P_{\omlog} =2\pi i \Res_{\alpha}(\omlog) T_D
\end{equation}
where $T_D$ is the current of integration over $D$,  and $\Res_{\alpha}(\omlog)$ is a complex valued function on $D$, called the \emph{residue} of $\omlog$ on $D$. The above computations show that  on each component $D_j$ of $D$ defined by $\{f_j=0\}$, we have 
\begin{equation}\label{eq:Res}
\Res_{\alpha}(\omlog)_{\mid D_j}= (e^{-h}g_j)_{|_{D_j}}.
\end{equation}
It follows from \eqref{log-current} and \eqref{eq:Res} that on $D_j$, $\Res_{\alpha}(\theta)_{\mid {D_j}}$ satisfies $\bar \partial_{-\alpha} \left( \Res_{\alpha}(\theta)_{|_{D_j}}\right)=0$, i.e. is an element of $H^0\left(D_j, \left({\cal L}_{\alpha}\right)_{|_{D_j}}\right)$ (see \eqref{dolbeault}).
 If  $\alpha=0$ and  $D_j$ is compact, then $\Res_{0}(\theta)_{|_{D_j}}= a_j$ is a constant. If  $D_j$ is a rational curve, then the flat $\C^{*}$ bundle $\left({\cal L}_{\alpha}\right)_{|_{D_j}}$ is trivial, and therefore $\Res_{\alpha}(\theta)_{|_{D_j}} = e^{h_j} a_j$ for some complex valued function $h_j$ on $D_j$. 

What will be important for us is that in the case when  $\cal L_{a} = L_{a}\otimes \C$ is a flat line bundle corresponding to a \emph{real} deRham class $a\in H^1_{dR}(S, \R)$,  there is a well-defined notion of positivity of the real and imaginary parts of $T_{\omlog}$,  independent of the representative $\alpha$ of $a\in H^1_{dR}(S, \R)$.
\begin{defn}\label{log-form-positive} Let $\omlog\in H^0(S, \Omega^1(\log D) \otimes \cal L_{a})$ be a logarithmic $1$-form on $S$ with values in the flat line bundle $\cal L_{a}$ corresponding to a deRham class $a\in H^1_{dR}(S, \R)$ via \eqref{H1}. We say that $\omlog$ is of \emph{positive type} if for some (and hence any) $\alpha \in a$, $\Res_{\alpha}(\omlog)$ is a real valued function which is  positive on each component $D_j$ of $D$. In particular, the current $\Im(T_{\omlog})$ is weakly positive and supported on $D$.
\end{defn}

\subsection{Twisted logarithmic 1-forms on class VII surfaces} We now specialize to class VII  complex surfaces. We recall that in this case  \eqref{Pic-identification} holds, i.e. the second morphism in \eqref{H1C} is an isomorphism. For this reason we shall refer to a topologically trivial,  flat $\C^*$-bundle $\cal L$ simply as a 
{\it flat line bundle}.  We  will  tacitly use the identification $H^1_0(S, \C^*) \cong \C^*$  and for a $\mu \in \C^*$ we write $\cal L_{\mu}$ for the corresponding  flat $\C^*$-bundle.  Furthermore,  $\cL=\cal L_{\mu} =  \cal L_{\alpha}$ for some closed complex-valued $1$-form $\alpha$.   The deRham class $[\alpha]$ is real iff $\mu>0$ and $[\alpha] <0$ iff $\mu >1$, according to the convention in Remark~\ref{r:convention}. In these terms, our task is to describe the minimal complex surfaces in the Kodaira class VII$_0$ which admit a non-trivial section 
\begin{equation}\label{eq:tlf}
0\neq \theta \in H^0(S,\O^1(\log D)\ot \cL_a)\neq 0, \end{equation}
for some divisor $D$ and  a flat line bundle $\cal L_{a}$ associated to a \emph{real} deRham class $a\in H^1_{dR}(S, \R)$, or equivalently, via \eqref{H1C}, to  a flat line bundle $\cal L_{\mu}$ with $\mu$ real and positive. 

\bigskip
We first study the question on the known complex surfaces in the class ${\rm VII}_0$.

\subsubsection{Twisted logarithmic $1$-forms on known class ${\rm VII}_0$ complex surfaces}\label{s:TLF}
\begin{lem} \label{Hopf mu>0} Let $S$ be a  Hopf surface satisfying \eqref{eq:tlf} for a divisor $D$ and and a flat line bundle $\cL_a$. Then $a \ge 0$. If, furthermore,  $S$ is a primary Hopf surface, then \eqref{eq:tlf} holds for $a=0$ and an elliptic curve $D$.  
\end{lem}
\begin{proof}  

A primary Hopf surface is the quotient of $\C^2\setminus \{0\}$ by the cyclic group  generated by  a contraction of  the type
\begin{equation}\label{Hopf-contraction}
F(z_1,z_2)=(\alpha z_1+tz_2^m, \beta z_2), \quad 0<|\alpha|< |\beta|<1,\quad t(\alpha-\beta^m)=0, \ m\ge 1.
\end{equation}
The secondary Hopf surfaces are obtained by taking a further quotient with a finite group, see \cite{kato-hopf}. To prove the statement, it is  thus enough to assume that $S$ is a primary Hopf surface  given by \eqref{Hopf-contraction}

Clearly,  \eqref{Hopf-contraction} preserves the $1$-form $\frac{dz_2}{z_2}$. The latter gives rise to a non-zero element of $H^0(S,\Omega^1(\log D))$ with $D$ corresponding to the elliptic curve $z_2=0$. We can thus  realize the value $a=0$. 

\smallskip
Suppose now $\theta \in H^0(S,\Omega^1(\log D)\otimes \cL_\mu)\neq 0$ with $\mu >1$. We shall derive that $\theta=0$.  The twisted logarithmic form $\theta$  is induced by a germ of logarithmic 1-forms (still denoted $\theta$) on $(\mathbb C^2,0)$,  such that $F^*\theta=\mu \theta$. We consider the following 3 cases, depending of the form of \eqref{Hopf-contraction}:
\begin{enumerate}
\item[(i)] Case $t\neq 0$.  Then there exists only one elliptic curve on $S$,  given by $\{z_2=0\}$, and 
$\theta$ has the form 
$$\theta(z_1,z_2)=A(z_1,z_2)dz_1+B(z_1,z_2)\frac{dz_2}{z_2},$$
where $A,B$ are germs of holomorphic functions on $(\mathbb C^2,0)$  satisfying
\[ A(F(z))\alpha = \mu A(z), \qquad mA(F(z))tz_2^{m-1}+\frac{B(F(z))}{z_2} =  \displaystyle\mu \frac{B(z)}{z_2}. \]
As $|\alpha|, |\beta|<1$ and $\mu>1$,  it follows that $A=0$; then the remaining  equation is $B(F(z))=\mu B(z)$, which yields $B=0$.
\item[(ii)] Case $t=0$ and $\alpha^k\neq \beta^l$, $k,l\in\mathbb N$.  In this case $S$ admits only two elliptic curves \cite[Proposition~18.2]{bpv}, corresponding to $z_1=0$ and $z_2=0$. Therefore
$\theta$ has the form 
$$\theta(z_1,z_2)=A(z_1,z_2)\frac{dz_1}{z_1}+B(z_1,z_2)\frac{dz_2}{z_2},$$
 where $A, B$ are holomorphic and satisfy
\[  A(F(z))\alpha =  \mu A(z)\qquad
B(F(z))\beta =  \mu B(z). \]
Using $0<|\alpha|<|\beta|<1<\mu$, we derive again $A=B=0$.
\item[(iii)] Case $t=0$ and $\alpha^k=\beta^l$ for  $k,l\in\mathbb N$.  In this case the Hopf surface $S$ is elliptic fibered over $\mathbb P^1(\mathbb C)$ by the meromorphic function  $z_1^kz_2^{-l}$ ~\cite[Proposition~18.2]{bpv}. There are no other curves in $S$ than the fibers of this fibration,  because the algebraic dimension  is  $a(S)=1$. Thus, any finite number of  polar curves for $\theta$ have the equations 
$$z_1^k-x_i z_2^l=0, \quad i=1,\ldots,p.$$
Thus, we get 
$$\theta= \sum_{i=1}^p A_i(z) \frac{d(z_1^k-x_i z_2^l)}{z_1^k-x_i z_2^l}.$$
The condition $F^*\theta=\mu\theta$ is then satisfied if and only if $A_i(F(z))=\mu A_i(z)$, for $i=1,\ldots,p$. It is easy to see that $A_i=0$, $i=1,\ldots,p$.
\end{enumerate}
\end{proof}

\begin{rem} The proof of Lemma~\ref{Hopf mu>0} shows that  for any primary Hopf surface $S$, the $1$-form  $z_2^k dz_2, \, k\in \N$ on $\C^2$ give rise to  a twisted holomorphic $1$-form (i.e. with $D=\emptyset$) with values in the flat line bundle $\cal L_{\beta^{k+1}}$;  in the case when $\beta \in \R$ these give rise to infinitely many values of $a>0$ and $D=\emptyset$.
\end{rem}

\begin{lem} \label{I-B1-forme} Let $S$ be  an Inoue--Bombieri surface. Then the condition \eqref{eq:tlf} holds only for $D=\emptyset$ and  a unique value of $a$. Furthermore, $a<0$ and  $\tauc(S)=\{a\}$.
\end{lem}
\begin{proof}  Recall that \cite{I74} Inoue--Bombieri surfaces  admit no curves, so $D=\emptyset$.  Furthermore, any Inoue--Bombieri surface $S$ is the quotient of $\bb H\times \bb C$ by a discrete subgroup $G$ of automorphisms, and we denote by $w=w_1 + i w_2, w_2>0$  and $z= z_1 + i z_2$  the complex coordinates on $\bb H$ and $\bb C$, respectively.
In the  notation of  \cite{I74},  there are the following three types of Inoue--Bombieri surfaces:
\begin{itemize}
\item $S$ of type $S_M=\bb H\times \bb C/G_M$. 
Then the $1$-form $dw$ on $\bb H\times\bb C$ induces a twisted $1$-form in $H^0(S_M, \O^1\ot \cal L_\mu)$ with $\mu$ real and $\mu>1$. Notice that in this case,  there is another twisted holomorphic $1$-form,  given by  $dz$,  i.e.  $H^0(S,\Omega^1\otimes \cL_\beta)\neq 0$ with $\beta\not\in\mathbb R$.
\item  $S$ of type $S^{(+)}_{N,p,q,r;t}=\bb H\times \bb C/G^{(+)}_{N,p,q,r;t}$. Then the normal subgroup $\G=\langle g_1,g_2,g_3\rangle$ of $G^{(+)}_{N,p,q,r;t}$ gives rise to the covers
$$\bb H\times \bb C\to \bb H\times \bb C/\G\to S^{(+)}_{N,p,q,r;t}.$$
Then $dw$ induces a holomorphic $1$-form on $\bb H\times \bb C/\G$ and a twisted holomorphic $1$-form in  $H^0(S^{(+)}_{N,p,q,r;t}, \O^1\ot \cal L_\mu)$, with $\mu>1$.
\item $S$ of type $S^{(-)}_{N,p,q,r}$. We have a double covering 
$$S^{(+)}_{N^2,p_1,q_1,r;0}\to S^{(-)}_{N,p,q,r}.$$
In this case a twisted holomorphic $1$-form in $H^0(S^{(+)}_{N^2,p_1,q_1,r;0},\O^1\ot \cal L_{\mu^2})$ with $\mu>1$,  induces a twisted holomorphic $1$-form in $H^0(S^{(-)}_{N,p,q,r},\O^1\ot\cal L_\mu)$.
\end{itemize}
In \cite[Lemma~4.4]{lcs1},  it is shown that  $\alpha:=dw_2/w_2 = d \log(w_2)$  defines a closed $1$-form on $S$ (in all of the cases above), such that the deRham class $a=[\alpha]$ satisfies  $\tauc(S)=\{a\}$. We notice that in the above constructions, we have $\cal L_{\mu} = \cal L_{a}$ in terms of \eqref{H1C}.

\smallskip We now argue that if $H^0(S,\Omega^1\otimes \cL_\mu)\neq 0$ and $\mu \in \R^*$, then $\mu$ is unique.  Indeed, if we have $0\neq \theta_1=dw\in H^0(S,\Omega^1\otimes \cL_{\mu})$ and $0\neq \theta_2\in H^0(S,\Omega^1\otimes \cL_\nu), \, \nu \in \R^*, \nu \neq \mu$, then $\omega:=\theta_1\wedge \theta_2\in H^0(S,K_S\otimes \cL_{\mu\nu})$ is non-zero. Otherwise $\theta_1/\theta_2$ would define a non-constant meromorphic function on $S$, an absurd.  As there are no curves on $S$ (see e.g. \cite{I74} or \cite{bpv}), we have  that $\omega$ does not vanish on $S$, i.e. $K_S\otimes \cL_{\mu\nu}=\mathcal O$.  The flat factor $\cL_{\mu \nu}$ such that the latter holds is unique (as $S$ admits no curves). On the other hand, inspecting the action of $G$ on the holomorphic $2$-form $dw\wedge dz$ on $\bb H \times \bb C$ according to the three types of Inoue--Bombieri  surfaces, we get 
\begin{itemize}
\item for $S_M$, $K_S \otimes \cal L_{\mu \beta}= \cal O$ with $\mu\beta \in \C^*\setminus \R^*$, an absurd.
\item for $S^{(+)}$, $K\otimes \cL_\mu=\mathcal O$, therefore $\nu =1$ an absurd as $H^0(S, \Omega^1)=0$ (see \cite{bpv}).
\item for $S^{(-)}$, $K\otimes \cL_{\mu}=\mathcal O$, therefore  $\nu =1$ an absurd as $H^0(S, \Omega^1)=0$ (see \cite{bpv}).
\end{itemize}
\end{proof}

\begin{lem}\label{l:Enoki} Let $S$ be an Enoki surface and $D_{r}$ its cycle of rational curves, corresponding to a flat line bundle $[D_r]=\cL_{t}, \, t\in\mathbb C^*$, $0<|t| < 1$. Then \eqref{eq:tlf} holds iff either $D=D_r$ and $\cL_a=\mathcal{O}$  or,  when $t$ is real, $D=\emptyset$ and $\cL_a=\cL_{t^{p}}$ where $p\ge 1$ is an integer. 
\end{lem}
\begin{proof} An Enoki surface can be defined by a contraction of the form~\cite{DK98}
$$F(z_1,z_2)=\left( t^nz_1z_2^n+\sum_{i=0}^{n-1}a_it^{i+1}z_2^{i+1}, tz_2\right), \quad 0<|t|<1, \quad n=b_2(S).$$
The curve $\{z_2=0\}$ induces the unique cycle of rational curves $D_r$ in $S$, so that we have $[D_r]=\cL_t$. The $1$-form $\theta_0:=\frac{dz_2}{z_2}$ satisfies $F^*(\theta)=\theta$,  so it gives rise to a  non-trivial section in $H^0(S,\Omega^1(\log D_r))$. Similarly,  for any integer $p \ge 1$, the $1$-form $\theta_p := z_2^{p-1} dz_2$ gives rise to a non-trivial section in $H^0(S,\Omega^1\otimes \cL_{t^{p}})$.

When the coefficients $a_i$ of the contraction germ $F$ are not all zero, $D_r$ is also the maximal divisor of $S$ and we can use the results in \cite[Section 3]{Nakamura} or \cite[Lemma~5.4]{DO} to conclude that the constant multiples of the above sections are the only non-trivial solutions of \eqref{eq:tlf}, i.e we must have either $D= D_r$ and $a=0$,  or $D=\emptyset$ and $\mu = t^{p}$. Notice that in the latter case, as $0< t^{p} <1$, we have  $\cL_{t^{p}} = \cL_a$ with $a>0$.

When the coefficients $a_i=0$, $S$ also admits an elliptic curve $E$ corresponding to $\{z_1=0\}$, and the maximal divisor in this case is $D_r + E$. A non-zero section in $H^0(S,\Omega^1(\log D)\otimes \cL_{\mu})$ then must correspond to a  germ of $1$-forms $\theta=A(z) \frac{dz_1}{z_1} + B(z)\frac{dz_2}{z_2}$ with $A(z)$ and $B(z)$ holomorphic functions defined around $0\in \C^2$, such that $F^*(\theta)= \mu \theta$. Expanding $A(z)$ and $B(z)$ around $0$ and using the specific form of $F$, one finds again that the multiples of  $\theta_p$, $p\ge 0, p\in \N$ are the only possible solutions. This shows that  in this case too for \eqref{eq:tlf} to hold  we must have either $D= D_r$ and $a=0$,  or $D=\emptyset$ and $\mu = t^{p}$. \end{proof}

\begin{lem}\label{ex:log-holom}  Suppose $S$ is a hyperbolic Kato surface. Then the condition  \eqref{eq:tlf} holds if and only if $D$ is the maximal divisor of $S$,  and for at most two values of $a$. Furthermore,  there exists a unique value of $a$  satisfying $a<0$ and this value is an upper bound of $\tauc(S)$.
\end{lem}
\begin{proof} Consider first the case of an Inoue--Hirzebruch surface. In the  notation of  Example~\ref{ex:Inoue-Hirzebruch}, one can check that  
\[ \omlog:= \xi \frac{dz_1}{z_1} + \eta \frac{dz_2}{z_2} \]
gives rise to a  twisted  logarithmic $1$-form on $S$ with values in the flat bundle is $\cal L_{\lambda_1}$  with  $\lambda_1>1$. The 
corresponding deRham class $a\in (-\infty, 0)$ is  the explicit upper bound of the set $\tauc(S)$ found in Example~\ref{ex:Inoue-Hirzebruch}. More generally, by \cite[Section 2.2]{DO}, on an Inoue--Hirzebruch there are in fact  two twisted logarithmic 1-forms $\theta_i \in H^0(S,\Omega^1(\log D)\otimes \cL_{\lambda_i})$, $i=1,2$,
with $\lambda_1\lambda_2=\pm 1$. For these examples,  $D$ equals the maximal divisor of $S$  and the condition on $\lambda_i$ leads to a unique twisting  $\cL_{\mu}$  with  $\mu>1$,  i.e.  with  $a<0$.

Similarly, if $S$ is a  Kato surface of intermediate type, Example~\ref{ex:Kato-intermediate} shows that $\frac{dz_2}{z_2}$ gives rise to a twisted logarithmic $1$-form with pole along the maximal divisor of $S$ and values in $\cal L_{k}, \, k\ge 2$. Again, $\cal L_{k} = \cal L_{a}$ where $a\in (-\infty, 0)$ is the explicit upper bound for $\tauc(S)$ identified in Example~\ref{ex:Kato-intermediate}.

By \cite[Lemma~5.4]{DO},  the above describe \emph{all} possible twisted logarithmic $1$-forms on Kato surfaces up to multiplicative constants, which yields  the uniqueness statement for $D$ and $a<0$. \end{proof}

We summarize the above discussion in the following
\begin{cor} Suppose $S$ is a known minimal complex surface in the class ${\rm VII}$, different than a secondary Hopf surface. Then $S$ admits a  divisor $D$ and flat bundle $\cal L_{a}$ corresponding to a 
deRham class $a\in (-\infty, 0]$,  such that $H^0(S,\Omega^1(\log D)\otimes \cL_a)\neq 0$. Furthermore, $a$ is an upper bound of $\tauc(S)$.
\end{cor}
\begin{proof} See Lemmas \ref{Hopf mu>0}, \ref{I-B1-forme}, \ref{l:Enoki}, \ref{ex:log-holom}.
\end{proof}

\begin{rem}\label{r:positive-type} An inspection shows that the twisted logarithmic $1$-forms produced in Lemma~\ref{Hopf mu>0} with $a=0$,  Lemma~\ref{I-B1-forme}, Lemma~\ref{l:Enoki} with $\mu=1$, and Lemma~\ref{ex:log-holom} are all of positive type in the sense of Definition~\ref{log-form-positive}. \end{rem}

\subsubsection{Twisted logarithmic $1$-forms on general ${\rm VII}_0$ complex surfaces}  We now study the condition \eqref{eq:tlf} on a general minimal complex surfaces in the Kodaira class VII$_0$. Our main objective here is to establish the following rough classification 
\begin{prop}\label{p:rough-classification-tlf} Let $S$ be a minimal compact complex surface in the class {\rm VII}$_0$, which satisfies the condition \eqref{eq:tlf} for some divisor $D$ and a flat line bundle $\cal L_{a}, \, a  \in H^1_{dR}(S, \R)$. Then
\begin{enumerate}
\item[\rm (a)]  if $D=\emptyset$, then $S$ must be a Hopf surface,  an Inoue--Bombieri surface, or an Enoki surface;
\item[\rm (b)]  if $D\neq \emptyset$ and  $S$ is not a Hopf surface, then $D$ contains all rational curves of $S$,  and a cycle of rational curves in each of its connected components. In particular, $D$ has at most two connected components (in which case $S$ must be an Inoue--Hirzebruch surface) and each connected component of $D$ coincides with a connected component of the maximal divisor of $S$.
\end{enumerate}
\end{prop}
The proof of the above result will occupy the reminder of the subsection,  and is divided in two steps, corresponding to the Lemmas below. 

We start with characterizing the case when $D=\emptyset$ in \eqref{eq:tlf}. The argument works for any flat bundle $\cal L= \cal L_{\alpha}, \, [\alpha] \in H^{1}_{dR}(S, \C)$.
\begin{lem}\label{D=0} Let $S$ be a minimal complex surface in the class {\rm VII}$_0$. There exists a non-trivial  twisted holomorphic $1$-form $\theta\in H^0(S,\O^1\otimes \cal L)$ only if $S$ is a Hopf surface,  an Inoue--Bombieri surface, or an Enoki surface.
\end{lem}
\begin{proof} If $b_2(S)=0$, then $S$ is either a Hopf surface or an Inoue--Bombieri surface by \cite{andrei}. 
 If $b_2(S)>0$, the claim follows from \cite[Theorem~5.6]{lcs1}.
\end{proof}

We now consider the case of $D\neq 0$. We assume that $b_2(S)>0$, i.e. $S$ is a minimal complex surface in the class {\rm VII}$_0^+$.
\begin{lem}\label{l:logarithmic-forms} Let $S$ be a surface in the class {\rm VII}$_0^+$ and suppose that  $S$ admits a non-trivial  logarithmic $1$-form $0\neq \omlog\in H^0(S,\O^1(\log D)\ot \cal L)$  with pole along a divisor $D\neq \emptyset$ and values in a flat line bundle  $\cal L=\cal L_{\alpha}$ for some closed complex-valued one form $\alpha$.  Then:
\begin{enumerate}
\item[\rm (a)] $d_{-\alpha} \omlog=0$ on $S\setminus D$, i.e. $d_{\cal L} \omlog$  is  the zero section of $H^0(S,\O^2(\log D)\ot \cal L)$. 
\item[\rm (b)] Each connected component of $D$ contains a cycle of rational curves $\Gamma$. In particular, $D$ has at most two connected components,  in which case $S$ is an Inoue--Hirzebruch surface.
\item[\rm (c)]  Each connected component of  $D$  coincides with a connected component of the maximal divisor of $S$. 
\end{enumerate}
\end{lem}
\begin{proof} 
(a) If $\omlog\in H^0(S,\O^1(\log D)\ot {\cal L}_{\alpha})$,  then $d_{-\alpha}\omlog= \partial_{-\alpha} \omlog \in H^0(S,\O^2(\log D)\ot {\cal L}_{\alpha})$ by \eqref{twisted-log}. If $d_{-\alpha}\omlog\neq 0$,  it gives rise to an effective divisor $\D$ with  $[\D]=K_S+[D]+\cal L$. Notice that $D-\D\neq 0$ because $c_1(K_S)^2=-b_2(S) \neq 0 =c_1(\cL)^2$. It follows that  and $D-\D$ is a numerically anticanonical divisor \cite{Dl}, and  therefore $S$ is a Kato surface by the results in that reference. On a Kato surface,  twisted logarithmic $1$-forms are classified in \cite{DO} and, by inspection, they are all $d_{\cal L}$-closed. By the discussion in Section~\ref{s:twisted-differentials},  this is equivalent to $d_{-\alpha} \omlog=0$, a contradiction. We thus conclude that  $d_{-\alpha}\omlog =0$ on $S\setminus D$ for the chosen flat connection $\nabla^{\alpha}$ on $\cal L=\cal L_{\alpha}$.

(b)   Let $D'$ be a connected component of the polar set $D$ of $\omlog$.  By Lemma~\ref{l:Enoki}, if $S$ is an Enoki surface, $D'=D=D_r$ is the cycle of rational curves and the claim is verified. Otherwise the intersection matrix of $D'$  is negative definite by \cite{enoki}.  We take a simply connected neighbourhood $V'$ of $D'$, which is  sufficiently small to be contractible onto a Stein space $V$ with an isolated singular normal point $x\in V$. Let $\pi:V'\to V$ be this contraction, $D'=\pi^{-1}(x)$ be the exceptional divisor, and $U:=V\setminus\{x\}$.  If $D'$ does not contain a cycle, $D'$ is simply connected then the restriction of $\cal L$ to $V'$ is trivial. By a result of Steenbrinck--van Straten \cite[Corollary~1.4]{SVS85} the mapping induced by the differentiation
$$d:H^0(U,\O^1_U)/H^0(V',\O^1_{V'})\to H^0(U,\O^2_U)/H^0(V',\O^2_{V'}(\log D'))$$
is injective, which gives a contradiction.  If there are two cycles in $D$, we apply \cite[Theorem 8.1]{Nakamura}.

(c) Let $C_1$ be an irreducible component of $D$ and suppose that there is a rational curve $C_2$,  which is not contained in $C_1$,  and such that $C_1\cap C_2=\{p\}$. There are holomorphic coordinates in a neighbourhood of $p$ with  $C_i=\{z_i=0\}$, $i=1,2$,  and  a locally defined smooth function $h$ and holomorphic functions $g_i$,  such that 
$$e^{h}\omlog= g_1\frac{dz_1}{z_1}+g_2 dz_2.$$
Since
$$0=d \left(e^{h}\omlog\right)=\frac{\part g_1}{\part z_2}dz_2\w \frac{dz_1}{z_1}+ \frac{\part g_2}{\part z_1}dz_1\w dz_2=\left(z_1\frac{\part g_2}{\part z_1}-\frac{\part g_1}{\part z_2}\right)\frac{dz_1}{z_1}\w dz_2$$
we see that $g_1$ must have the form
$$g_1(z_1,z_2)=\b(z_1)+z_1z_2\g(z_1,z_2).$$
If $\b(0)\neq 0$, then the restriction of $\omlog$ to $C_2$ has a simple pole at $p$, which is impossible. Therefore $\b(0)=0$ and
$$e^{h}\omlog=\left(\frac{\b}{z_1}+ z_2\g\right)dz_1+g_2 dz_2.$$
The above form has no pole along $C_2$,  a contradiction. Thus $D$ contains the whole connected component of $\G$. \end{proof}

\begin{proof}[\bf Proof of Proposition~\ref{p:rough-classification-tlf}] The claim in (a) follows from Lemma~\ref{D=0}. For (b), notice that, by Lemma~\ref{I-B1-forme},  $S$ cannot be an Inoue--Bombieri surface; as $S$ is not a Hopf surface by assumption, it must be a surface in the class  {\rm VII}$_0^+$ by \cite{andrei}. We are thus under the hypotheses of Lemma~\ref{l:logarithmic-forms}. In order to establish (b), we only need to show that $D$ contains \emph{all} rational curves of $S$.  By \cite{Nakamura}, $S$ has only rational curves, unless it is an Enoki surface. In the case of an Enoki surface, the claim follows from Lemma~\ref{l:Enoki}. Otherwise, we need to show that $D$ equals the maximal divisor of $S$. If this is not the case, by Lemma~\ref{l:logarithmic-forms},  the maximal divisor $D_m$ must be of the form $D_m= D + C$, where $C$ denotes the other connected components of $D_m$. As $D$ contains a cycle (again by Lemma~\ref{l:logarithmic-forms}), by \cite[Theorem 1.4]{D21}, $C$ must also contain a cycle, and we conclude by \cite{Nakamura} that $S$ is then an Inoue--Hirzebruch surface. In this case,  $D_m=D$ by Lemma~\ref{ex:log-holom}, a contradiction. We thus conclude $D=D_m$, so $D$ contains all the rational curves on $S$. \end{proof}

We next single out the case when the  flat line bundle is trivial, i.e. $a=0$ in \eqref{eq:tlf}.
\begin{lem} \label{trivialtwisting} Let $S$ be a minimal compact complex surface in the class {\rm VII}$_0$. There exists a non-trivial logarithmic $1$-form $\theta\in H^0(S,\O^1(\log D))$ if and only if $S$ is a Hopf surface or an Enoki surface.
 \end{lem}
\begin{proof} By Lemmas~\ref{Hopf mu>0} and \ref{l:Enoki}  the primary Hopf surfaces and all Enoki surfaces admit non-trivial logarithmic 1-forms, i.e  \eqref{eq:tlf} holds with $\cal L_{\mu}=\cal O$.

\smallskip Conversely, suppose $H^0(S, \O^1(\log D))\neq 0$. As $H^0(S,\O^1)=0$ for any surface in class VII$_0$, $D \ne \emptyset$. If $D$ has more than one connected components, it must have precisely two connected components and $S$ must be an  Inoue--Hirzebruch surface (see \cite[Theorem~4.22]{D21}). By  \cite{DO},   such a  surface has only twisted logarithmic 1-forms, i.e. $H^0(S, \O^1(\log D))= 0$. Thus, $D$ is connected.  If the intersection matrix of $D$ is negative definite, $D$ can be contracted onto a normal point $x$. Let $p:S\to S'$ be this contraction.
 Let $V'$ be a Stein neighbourhood of $x$ and $V=p^{-1}(V')$, $U'=V'\setminus\{x\}$, $U=V\setminus D$.   By the theorem of Steenbrinck--Van Straten \cite[Corollary~1.4]{SVS85} the mapping induced by the differentiation
$$d:H^0(U',\O^1_{U'})/H^0(V,\O^1_{V})\to H^0(U',\O^2_{U'})/H^0(V,\O^2_{V}(\log D))$$
is injective. Since $[d\theta]=0$, we obtain a contradiction.
Thus, the intersection matrix of $D$  cannot be definite, and  therefore there is a divisor $A$ with support in $D$ such that $A^2=0$.  If $b_2(S)=0$, it follows that $S$ is a Hopf surface; if $b_2(S)>0$, Enoki's theorem \cite{enoki} shows that $S$ must an affine bundle over an elliptic curve. Such a surface is an  Enoki (Kato) surface,  by the construction of generic contractions in \cite{DK98}. \end{proof}

\begin{cor} \label{c:unknown-tlf} Suppose $S$ is a \emph{unknown} minimal compact complex surface in the class {\rm VII}$_0$,  satisfying the condition \eqref{eq:tlf} for some divisor $D$ and a flat line bundle $\cal L_{a}, \, a  \in H^1_{dR}(S, \R)$. Then  $a\neq 0$, $D$ equals the maximal divisor of $S$ and has one connected component. Furthermore,  all of the irreducible components of $D$ are either smooth rational curves or rational curves with a double point. \end{cor}
\begin{proof} By Proposition~\ref{p:rough-classification-tlf} and Lemma~\ref{trivialtwisting}, we can assume that $a\neq 0$ and $D\neq \emptyset$ is connected and contains all rational curves on $S$. By \cite{Nakamura}, $D$ must coincide with the maximal divisor.  This reference also gives the possible type of rational curves. The fact that $D$ contains a cycle of rational curves  follows from Proposition~\ref{p:rough-classification-tlf}. \end{proof}

\subsubsection{The associated foliation}   We start with the following observation
\begin{lem}\label{l:isolated-zero}   Let $S$ be a surface in the class {\rm VII}$_0^+$ and suppose that  $S$ admits a non-trivial  logarithmic $1$-form $0\neq \omlog\in H^0(S,\O^1(\log D)\ot \cal L)$  with pole along a divisor $D\neq \emptyset$ and values in a flat line bundle  $\cal L=\cal L_{\alpha}$ for some closed complex-valued one form $\alpha$. Then the analytic set of points where $\omlog$ vanishes is either empty or consists of isolated points. 
\end{lem}
\begin{proof} If there is a divisor $A\subset \{\omlog=0\}$, its connected component has an empty intersection with the pole $D$ by Lemma~\ref{l:logarithmic-forms}. It follows that there are at least two connected components of curves on $S$, one of which contains a cycle of rational curves by Proposition~\ref{p:rough-classification-tlf}. We conclude by \cite{D21} and \cite{Nakamura} that $S$ must be either an Enoki surface or an Inoue--Hirzebruch surface.  In both cases $A =\emptyset$ by the description of $\theta$ in Lemmas~\ref{l:Enoki} and ~\ref{ex:log-holom}. \end{proof}

Using the local form \eqref{log-1-form} (in complex dimension $2$) and Lemma~\ref{l:isolated-zero},  any non-trivial twisted logarithmic $1$-form $\theta$ defines a (possibly singular) holomorphic foliation $\cal F$ on $S$ (see e.g. \cite{B00}). The singularities of $\cal F$ are necessarily at the isolated zeroes and the pole $D$ of $\theta$.  We recall the following basic
\begin{defn}\label{d:singularity-foliation} Let $\cal F$ be a holomorphic foliation on a complex surface $S$,  defined in a local chart around a singularity $p$ by a vector field $V(z_1,z_2)$,  $V(p)=0$. Denote by  $\l_1$ and $\l_2$ the eigenvalues of $DV(p)$. The singularity $(\cal F,p)$ is called  \emph{non-degenerate}  if $\l_1\l_2\neq 0$. We notice that at any non-degenerate singularity $p$, the numbers  $\l(p)=\frac{\l_1}{\l_2}$  and  $\l^{-1}(p) = \frac{\l_2}{\l_1}$ are invariants of $(\cal F, p)$, independent of the chosen vector field $V$, and are called \emph{characteristic numbers} of the non-degenerate singularity of $\cal F$.
\end{defn}


\begin{lem}\label{singularity-foliation} Let $S$ be a minimal surface in the class {\rm VII}$_0^+$,  which is not an Inoue--Hirzebruch or Enoki surface. Suppose that  $S$ admits a non-trivial  logarithmic $1$-form 
$$0\neq \omlog\in H^0(S,\O^1(\log D)\ot \cal L_\alpha), \qquad  [\alpha] \in H^1_{dR}(S, \R),$$
 with pole along a non-trivial connected divisor $D=\sum_i D_i$.  Let $\cal F$ be the foliation defined by $\omlog$. Then
\begin{itemize}
\item[(i)]  The intersections of curves in $D$ are  non-degenerate singularities of $\cal F$ with characteristic numbers $\l^{\pm 1}(p)\in \bb R^*$,
\item[(ii)]  If $D_i\cap D_j\neq \emptyset$,  then $\left(\frac{Res_\alpha(\theta)_{\tilde D_j}}{Res_\alpha(\theta)_{\tilde D_i}}\right)\in \mathbb R^*$, see \eqref{eq:Res}.
\end{itemize}
\end{lem}
\begin{proof} Using a result of Nakamura~\cite[Theorem~10.2]{Nakamura}, under the hypotheses of the Lemma, we know that each component $D_j$ is either smooth rational curve, or is a rational curve with a double-point singularity. Let $\hat\pi:\hat S\to S$ be the minimal $\Z$-cover, see Definition~\ref{d:cyclic}, and $\hat D=\hat\pi^{-1}(D)$, $\hat\a:=\hat\pi^*\a$, $\hat\a=d\hat h$, $\hat\omlog=\hat\pi^*\omlog$.  Each component $\hat D_j$ is again a rational curve. In a local chart $\hat U$ around $\hat D_1 \cap \hat D_2$ we have (see\eqref{twisted-log})
$$e^{\hat h}\hat\omlog=\omlog_0+g_1(z)\frac{dz_1}{z_1}+g_2(z)\frac{dz_2}{z_2}$$
where $\omlog_0$, $g_1$, $g_2$ are holomorphic.  As $\hat D_i$ are compact curves,  it follows that the restriction of $g_1(0, z_2)$ on $\hat D_1$ is a constant, say $a_1$, and similarly, the restriction of $g_2(z_1, 0)$ to $\hat D_2$ is a constant $a_2$. Therefore
$$Res_{\hat\alpha}(\hat \theta)_{\mid z_1=0}=e^{-\hat h} g_1(0,z_2)= e^{-\hat h} a_1, \qquad Res_{\hat\alpha}(\hat \theta)_{\mid z_2=0}=e^{-\hat h} g_2(z_1,0)= e^{-\hat h} a_2.$$
Notice that if $a_1=0$ (resp. $a_2=0$) then $\hat D_1$ (resp. $\hat D_2$) will not be a pole, a contradiction. It thus follows that the holomorphic foliation $\cal F$ defined by $\theta$  has a non-degenerate singularity at $p\in D_1 \cap D_2$, with corresponding characteristic numbers $(\lambda, \lambda^{-1}):=(-\frac{a_1}{a_2}, -\frac{a_2}{a_1})$, and, moreover 
\[ \frac{a_1}{a_2} = \frac{Res_\alpha(\theta)_{\tilde D_1}}{Res_\alpha(\theta)_{\tilde D_2}}.\]
 We next compute the Camacho--Sad indices  $CS$ of the foliation $\cal F$ along the curves $D_i$ of $D$ (see \cite{CS82}):
\begin{equation}\label{comput}
\begin{split} 
CS(\cal F,D_2,D_2\cap D_1) &:=  Res_{z_1=0}\left(\frac{\part}{\part z_2}(-\frac{z_2g_1}{z_1g_2})_{\mid D_2}\right)  \\
                                             &= Res_{z_1=0}\left( -\frac{g_1(z_1,0)}{z_1g_2(z_1,0)}\right)\\
                                             &= -\frac{g_1(0,0)}{g_2(0,0)}=-\frac{a_1}{a_2} =-\frac{Res_\alpha(\theta)_{\tilde D_1}}{Res_\alpha(\theta)_{\tilde D_2}}\\
                                             &= \lambda(p).
\end{split}\end{equation}
Recall that by \cite{CS82}, $\cal F$ admits a \emph{separatrix} through its non-degenerate singularity $p=(0,0)$, i.e. there exists an $\cal F$-invariant complex analytic curve passing through $(0,0)$. Looking at the classification of non-degenerate singularities of foliations (see e.g. \cite[pp. 2-3]{B00}),  the only possibility with $\l=\l(p)\not\in\bb R$ are singularities in the  Poincar\'e domain with two separatrices $\{z_i=0\}$, $i=1,2$. The holonomy of each separatrix  is generated by the contraction (or its inverse) $z_i\mapsto \mu z_i$ with $\mu=exp(2\pi i \l^{\pm 1})$. As the holonomy group is  a contraction in this case, we see that any leaf $L$ accumulates on $D$ and all intersection points $p$ have complex non-real characteristic numbers $\l^{\pm 1}(p)$. 
Furthermore,  any  leaf $L$ is closed in $S\setminus D$. In particular $\bar L$ is not analytic. Under the hypotheses of the Lemma and Enoki theorem~\cite{enoki}, the intersection matrix of the components $D_i$ of $D$ is negative definite. We can then contract $D$ onto a normal singular point $f: S\to X$  by Grauert theorem (see e.g. \cite[Theorem~4.9]{Laufer71}). Since the Hausdorff measure of dimension one of a point is zero,  a theorem of Remmert--Stein--Shiffman \cite{schiffman68} shows that $M=\overline{f(L)}$ is an analytic curve. But then $\bar L$ is contained in the analytic curve $f^{-1}(M)$,  a contradiction.  As $a_1=g_1(0)\neq 0$ and $a_2=g_2(0)\neq 0$, we thus conclude $\lambda(p) \in \R^*$ and  $\frac{Res_\alpha(\theta)_{\tilde D_i}}{Res_\alpha(\theta)_{\tilde D_j}}= -\lambda(p) \in \R^*$ (see \eqref{comput}). \end{proof}

\begin{lem}\label{exact-current} Under the hypotheses of Lemma~\ref{singularity-foliation}, suppose that $\cal L=\cal L_\alpha$ for a \emph{real} deRham class $0\neq [\alpha]\in H^1_{dR}(S, \R)$ and that the characteristic numbers at the intersection of curves of $D$ are \emph{negative reals}. Then, we can multiply $\theta$ by a non-zero complex number such that the residues $Res_{\alpha}(\omlog)$ defined in \eqref{eq:Res} are real valued and positive,  and the real part $\Re(P_{\omlog})$ of the current $P_{\omlog}$ associated to $\omlog$ via \eqref{omega-current}  is exact, i.e. satisfies $\Re(P_{\omlog})= d_{\alpha} \tau$ for a degree $0$ current $\tau$ on $S$. Furthermore,     \[d_{\alpha}d^c_{\alpha} \tau = 2\pi \Res_{\alpha}(\omlog)T_{D}>0,\]
and $[\alpha] \notin \tauc(S).$
\end{lem}
\begin{proof}  The argument at the beginning of the proof of Lemma~\ref{singularity-foliation} shows that  $\Res_{\alpha}(\omlog)_{|_{\tilde D_j}} \in H^0(D_j, \cal L_{\alpha})$ is of the form $e^{-h_j}a_j$ where $h_j$ is a the restriction of a smooth function to $D_j$ and $a_j \neq 0$ is a complex number. Let us fix one such component $D_j \subset D$. By Lemma \ref{singularity-foliation},  and using the connectedness of $D$, it follows that if we  multiply $\theta$ by the non-zero constant $1/a_j$, we  have that  $Res_\a(\omlog)$ is \emph{real-valued} and non-zero along  any other component $D_i$, whence on $D$. Furthermore, if all the characteristic numbers of $\cal F$ are negative, by \eqref{comput}, we may also assume that 
$\Res_{\alpha}(\omlog)>0$. 

We now notice that the RHS of \eqref{log-current} is $d_{\alpha}$-closed (as a current on $S$), which follows directly from the  formula \eqref{twisted-log} and the fact that $(g_j)_{|_{D_j}}=a_j$ are constants. It follows that  $Q:=\partial_{\alpha}P_{\omlog}$ is a $\bar\partial_{\alpha}$-closed current of bi-degree $(2,0)$, and thus is of the form $Q=T_{\psi}$ for a $(2,0)$-form $\psi$ on $S$ with $\bar \partial_{-\alpha}\psi=0$, i.e. $\psi \in H^{0}(S, K_S\otimes \cal L_{\alpha})$. As $S$ is in the class VII$_0^+$ (and $\cal L$ is flat) $\psi=0$ (see e.g. \cite[Remark~2.21]{Dl}), and thus $Q=0$.  We conclude that 
\[d_{\alpha} P_{\omlog} = \bar \partial_{\alpha} P_{\omlog} = 2\pi i \Res_{\alpha}(\omlog) T_D.\]
Taking real parts  yields $d_{\alpha} R =0$,  where $R:=\Re(P_{\omlog})$.  Thus, $R$ gives rise to a class $[R]\in H_1^{\alpha} (S, \R)\cong \big(H^1_{\alpha}(S, \R)\big)^*\cong H^3_{-\alpha}(S, \R)$, as discussed in Section~\ref{s:twisted-differentials}.
By Lemma~\ref{l:logarithmic-forms}-(b), $S$ contains a cycle of rational curves and,  therefore,  is diffeomorphic to $(S^1\times S^3)\sharp k \overline{\C P}^2$ by the results in \cite{N90}. It  follows from  \cite[Theorem 1.5]{lcs1} that either $\cal L =\cal O$ is trivial, or else $[R]=0$ and hence $R= d_{\alpha} \tau$.  As we have assumed $[\alpha] \neq 0$, the first possibility is impossible. 
We thus conclude that $d_{\alpha}d^c_{\alpha} \tau = 2\pi \Res_{\alpha}(\omlog)T_{D}>0$ for a degree zero (real valued) current $\tau$. The  last claim then follows from Propositions~\ref{p:conformal-dual}  and \ref {p:lcs-characterization}. \end{proof}


\begin{proof}[{\bf Proof of Theorem~\ref{thm:logarithmic-forms}}]  If $b=0$, we must be in the situation described in (a) by Lemma~\ref{trivialtwisting} and Theorem~\ref{main} (b)\&(d). 
We suppose from now on $b<0$. If $b_2(S)=0$, by~\cite{andrei} $S$ is either a Hopf surface or an Inoue--Bombieri surface. The Hopf surfaces do not admit  twisted logarithmic $1$-forms with values in a flat line bundle $\cL_{b}$ with $b<0$ by Lemma~\ref{Hopf mu>0}. The non-zero twisted logarithmic $1$-forms on an Inoue--Bombieri surface have values in the flat bundle $\cL_a$ with $\tauc(S)=\{ a \}$ by Lemma~\ref{I-B1-forme}. We thus recover the case (c) of Theorem~\ref{thm:logarithmic-forms}, see Theorem~\ref{main} (a).

From now on we assume that $b_2(S)>0$ and $b\neq 0$, i.e. the surface is in the class ${\rm VII}_0^+$.  
We are thus in the setup of Lemma~\ref{l:logarithmic-forms}. 

If $D$ has two connected components, $S$ is an Inoue--Hirzebruch surface by Lemma~\ref{l:logarithmic-forms}-(b). In this case, there are two flat holomorphic line bundles $\cL=\cL_{\mu_i}, \, i=1,2$ with $H^0(S, \Omega^1(\log D)\otimes \cal L_{\mu_i})\neq 0$,   corresponds to a real numbers  $-1<\mu_1<1$ or  $\mu_2 >1$ \cite{DO}.  By Remark~\ref{convention1}, $\cL_{\mu_2}=\cL_{b}$ with $b<0$, and in this case, by Example~\ref{ex:Inoue-Hirzebruch} and Lemma~\ref{ex:log-holom},  $\tauc(S) \subset (-\infty, b)$, i.e. we obtain the case (b) of the Theorem.

We now suppose that $D$ has one connected component.  In this case, by Lemma~\ref{exact-current} we can assume without loss that ${\rm Res}_{\beta}(\theta)>0$ and  thus  $b=[\beta] \notin \tauc(S)$. Recall that  $\tauc(S)$ is an open interval by Theorem~\ref{thm:tau-structure}. Suppose $b$ is an upper bound of $\tauc(S)$. By Lemma~\ref{exact-current} again, there exists a degree $0$ curent $\tau$ on $S$ satisfying $T=d_{\beta} d^c_{\beta} \tau = 2\pi \Res_{\beta}(\omlog)T_{D}>0, \, T\neq 0$. We conclude by Proposition~\ref{p:sign} that $\tau <  0$. By Theorem~\ref{thm:brunella}, $S$ must be a hyperbolic Kato surface (we assume $b_2(S)>0$ at this point) and furthermore it must be of intermediate type as we supposed that the maximal divisor $D$ is connected. We recover again the case (b) of the Theorem.  The only remaining case is that $\tauc(S) \subset (b, \infty)$, a situation  which we  recorded in the case (d) of the Theorem. \end{proof}

\end{document}